\documentclass[11pt,a4paper,twoside]{amsart}
\usepackage{amsmath,amsfonts,amsthm,amsopn,color,amssymb}
\usepackage{enumitem}
\usepackage{palatino}
\usepackage{graphicx}
\usepackage[colorlinks=true]{hyperref}
\hypersetup{urlcolor=blue, citecolor=red, linkcolor=blue}
\usepackage[latin1]{inputenc}  

\def\Xint#1{\mathchoice
  {\XXint\displaystyle\textstyle{#1}}%
  {\XXint\textstyle\scriptstyle{#1}}%
  {\XXint\scriptstyle\scriptscriptstyle{#1}}%
  {\XXint\scriptscriptstyle\scriptscriptstyle{#1}}%
  \!\int}
\def\XXint#1#2#3{{\setbox0=\hbox{$#1{#2#3}{\int}$}
  \vcenter{\hbox{$#2#3$}}\kern-.5\wd0}}
\def\dashint{\Xint-}

\newcommand{\diva}{\underline{\mathbf{\al}}}
\newcommand{\divam}{\underline{\mathbf{\al}}_m}

\newcommand{\al}{\alpha}

\newcommand{\dt}{\delta}
\newcommand{\eps}{\varepsilon}

\newcommand{\lm}{\lambda}

\newcommand{\R}{\mathbb{R}}

\newcommand{\Sd}{\mathbb{S}^2}

\newcommand{\weakto}{\rightharpoonup}

\DeclareMathOperator{\diam}{diam}
\DeclareMathOperator{\dist}{dist}
\DeclareMathOperator{\distd}{d}
\DeclareMathOperator{\interior}{int}
\DeclareMathOperator{\meas}{meas}

 \DeclareMathOperator{\Id}{Id}

\renewcommand{\l }{\lambda}

\renewcommand{\O}{\Omega}

\newcommand{\N}{\mathbb{N}}

\def\bbm[#1]{\mbox{\boldmath $#1$}}
\newcommand{\beq }{\begin{equation}}
\newcommand{\eeq }{\end{equation}}

\newtheorem{theorem}{Theorem}[section]
\newtheorem{lemma}[theorem]{Lemma}

\newtheorem{proposition}[theorem]{Proposition}
\newtheorem{remark}[theorem]{Remark}

\def\sideremark#1{\ifvmode\leavevmode\fi\vadjust{\vbox to0pt{\vss
 \hbox to 0pt{\hskip\hsize\hskip1em
 \vbox{\hsize3cm\tiny\raggedright\pretolerance10000
  \noindent #1\hfill}\hss}\vbox to8pt{\vfil}\vss}}}%

\title[Existence and non existence results for the singular Nirenberg problem.]
{Existence and non existence results for the singular Nirenberg problem.
}
\author{ Francesca de Marchis and Rafael L\'{o}pez-Soriano}
\address{Francesca De Marchis, Dipartimento di Matematica, Università di Roma {\em Sapienza}, P.le Aldo Moro 5, 00185 Roma, Italy.}
\email{demarchis@mat.uniroma1.it}
\address{Rafael L\'{o}pez-Soriano, Departamento An\'{a}lisis Matem\'{a}tico, Universidad de Granada, Campus Fuentenueva, 18071 Granada, Spain.}
\email{rafals@ugr.es}
\thanks{}

\email{}
\date{}

\keywords{Prescribed Gaussian curvature problem, Conical singularities, Variational methods, Min-max schemes.}

\subjclass[2010]{35J20, 35R01, 53A30.}

\begin{document}

\maketitle

\begin{abstract}
In this paper we study the problem, posed by Troyanov in \cite{Troy}, of prescribing the Gaussian curvature under a conformal change of the metric on surfaces with conical singularities. Such geometrical problem can be reduced to the solvability of a nonlinear PDE with exponential type non-linearity admitting a variational structure. In particular, we are concerned with the case where the prescribed function $K$ changes sign.
When the surface is the standard sphere, namely for the singular Nirenberg problem, we give sufficient conditions on $K$, concerning mainly the regularity of its nodal line and the topology of its positive nodal region, to be the Gaussian curvature of a conformal metric with assigned conical singularities.
\

Besides, we find a class of functions on $\mathbb{S}^2$ which do not verify our conditions and which can not be realized as the Gaussian curvature of any conformal metric with one conical singularity. This shows that our result is somehow sharp.
\end{abstract}

\section{Problem}\label{Introduction}
\setcounter{equation}{0}
On a compact surface $(\Sigma,g)$ we consider the equation

\beq\label{eqdelta}
-\Delta_g v=\lm\left(\frac{K e^{v}}{\int_{\Sigma}Ke^{v}dV_{g}}-\frac1{|\Sigma|}\right)-4\pi\sum_{j=1}^m\alpha_j\left(\delta_{p_j}-\frac1{|\Sigma|}\right)\quad\;\mbox{in}\;\Sigma.
\eeq
Here $\Delta_g$ denotes the Laplace-Beltrami operator, $dV_g$ the volume element relative to the metric $g$ and $|\Sigma|$ the area of $\Sigma$, while $\lambda$ is a positive parameter, $K:\Sigma \to \mathbb{R}$ is a Lipschitz function, $\alpha_j>0$, and $\delta_{p_j}$ is the Dirac measure centered at point $p_j \in \Sigma$.

The analysis of \eqref{eqdelta} is motivated by the study of vortex type configurations in the Electroweak theory of Glashow-Salam-Weinberg \cite{Lai} and in Self-Dual Chern-Simons theories \cite{Dun}. In the monographs \cite{Tarbook,yang}, and in \cite{TarAnal} the reader can find further details and a wide set of references concerning these applications.

\

However, in this paper we focus on the geometric meaning of \eqref{eqdelta}, which appears in the prescribed Gauss curvature problem. Indeed if $\tilde g$ is a conformal metric to $g$ on $\Sigma$, namely $\tilde{g} =e^u g$, and $H_{\tilde{g}}$, $H_g$ are the Gaussian curvatures relative to these metrics, then the logarithm of the conformal factor satisfies the equation
\beq\label{presreg}
-\Delta_g u + 2 H_{g} =2 H_{\tilde{g}}e^u \quad\mbox{in}\quad \Sigma.
\eeq

For an assigned Lipschitz function $H$ defined on $\Sigma$, a classical problem is to find a conformal metric $\tilde{g}$ having $H$ as the Gaussian curvature: this amounts to solve \eqref{presreg} with $H_{\tilde{g}}=H$. The solvability of this problem depends on the Euler characteristic of $\Sigma$, $\chi(\Sigma)$. Actually, if $\chi(\Sigma) = 0$, the problem is completely solved in \cite{KazWar}; whereas for the case $\chi(\Sigma)<0$ there are necessary, \cite{KazWar}, and sufficient conditions, \cite{Ber,Aub,BGS,DePR}, but the problem is not completely settled.

The problem of prescribing Gaussian curvature on $\mathbb{S}^2$ endowed with the standard metric $g_0$, proposed initially by Nirenberg, is the most delicate case and the known results are partial, \cite{Aub,ChYg1,ChYg2,ChenLiAnn,ChenLins,ChenLiDCDS,KazWar}. For the Nirenberg problem, equation \eqref{presreg} can be reformulated as
\beq\label{Nirenberg}
-\Delta_{g_0} u=8\pi\left(\frac{H e^{u}}{\int_{\mathbb{S}^2}He^{u}dV_{g_0}}-\frac1{|\mathbb{S}^2|}\right)\quad\mbox{in}\quad \mathbb{S}^2,
\eeq
which corresponds to \eqref{eqdelta} with $\lambda=8\pi$ and $\alpha_j=0$ for any $j$.

Throughout the paper we will refer to \eqref{Nirenberg} as the \emph{regular} Nirenberg problem. More generally, problem \eqref{eqdelta} appears when one allows the conformal class to contain metrics that introduce conical-type singularities on $\Sigma$.\\ 
Following the pioneer works of Troyanov \cite{Troy,Troy2}, we say that $(\Sigma, \tilde{g})$ defines a punctured Riemann surface $\Sigma \setminus \{p_1, \ldots, p_m\}$ that admits a conical singularity of order $\alpha_j>-1$ at the point $p_j$ with $j=1, \ldots ,m$, if in a coordinate system $z = z(p)$ around $p_j$ centered at the origin, i.e. $z(p_j) = 0$, we have
$$
\tilde{g}(z)=|z|^{2\alpha_j}e^w|dz|^2,
$$
with $w$ a smooth function. In other words, as a differentiable manifold, $\Sigma$ admits a tangent cone with vertex at $p_j$ and total angle $2\pi(\alpha_j + 1)$, for $j = 1, \ldots, m$.

For a given Lipschitz function $K$, we seek a metric $\tilde{g}$, conformal to $g$ in $\Sigma \setminus \{p_1, \ldots, p_m\}$, namely $\tilde{g}=e^v g$ in the punctured surface, admitting conical singularities of orders $\alpha_j$'s at the points $p_j$'s and having $K$ as the associated Gaussian curvature. Analogously to the regular problem, we can reduce such geometrical question to the solvability of the differential equation
\beq\label{presregs}
-\Delta_g v + 2 K_{g} =2 K e^v -4\pi \sum_{j=1}^m \alpha_j \delta_{p_j} \quad\mbox{in}\quad \Sigma,
\eeq
where $K_g$ is the Gaussian curvature associated to the metric $g$.

Integrating \eqref{presregs} and applying the Gauss-Bonnet Theorem, one immediately obtains 
\beq\label{EulerchS}
2 \int_{\Sigma} Ke^v dV_g = 2 \int_{\Sigma} K_{g} dV_g+ 4\pi \sum_{j=1}^m \alpha_j=4\pi\left( \chi(\Sigma) + \sum_{j=1}^m \alpha_j\right) .
\eeq

As for the regular case, the solvability of \eqref{presregs} depends crucially on the value of the generalized Euler characteristic for singular surfaces, defined as follows
\beq\label{EulerchS2}
\chi(\Sigma,\divam)=  \chi(\Sigma) + \sum_{j=1}^m \alpha_j.
\eeq

In \cite{Troy} the case $\chi(\Sigma,\divam)\leq 0$ has been treated obtaining existence results analogous to the ones for the regular case, \cite{Ber,KazWar}.

It is worth to notice that for $\chi(\Sigma,\divam)>0$, \eqref{EulerchS} implies that, if \eqref{presregs} admits a solution, $K$ has to be positive somewhere. In \cite{Troy} it is proved that, if $0<4\pi \chi(\Sigma,\divam)<8\pi (1+\min_j\{\min{\{0,\alpha_j}\}\})$, then this necessary condition is also sufficient to obtain existence. In general, if $\chi(\Sigma,\divam)>0$, using \eqref{EulerchS}, it can be seen that \eqref{presregs} can be transformed into \eqref{eqdelta} with $\lambda = 4\pi \chi(\Sigma,\divam)$ and the extra term $\frac{4\pi\chi(\Sigma)}{|\Sigma|}-2K_g$ in the right hand side. Having $\frac{4\pi\chi(\Sigma)}{|\Sigma|}-2K_g$ zero mean value, this difference does not play any role and we will not comment on this issue any further.

\

Now we transform the equation \eqref{eqdelta} into another one which admits variational structure. Let $q \in \Sigma$ and $G(x,q)$ be the Green function of the Laplace-Beltrami operator on $\Sigma$ associated to $g$, i.e.
\beq\label{fgreen}
-\Delta_g G(x,q)=\dt_q-\frac{1}{|\Sigma|}\quad \mbox{in}\quad \Sigma,\qquad\quad\int_{\Sigma }G(x,q)dV_{g}(x)=0.
\eeq

Moreover, given $p_1,\ldots,p_m \in \Sigma$ and $\alpha_1,\ldots,\alpha_m \in (-1,+\infty)$ we define
\beq\label{accame}
h_m(x)=4\pi\sum\limits_{j=1}^m\al_jG(x,p_j)=-2\sum\limits_{j=1}^m\al_j\log dist(x,p_j)+h(x),
\eeq
where $h$ is the regular part of $h_m$.

The change of variable
$$
u= v + h_m
$$
transforms \eqref{eqdelta} into the equation
\beq\label{equation}
-\Delta_g u=\lm\left(\frac{\tilde K e^{u}}{\int_{\Sigma}\tilde K e^{u}dV_{g}}-\frac1{|\Sigma|}\right)\quad\mbox{in}\quad \Sigma,
\eeq
where 
\begin{equation}\label{tildeK}
\tilde K= K e^{-h_m}.
\end{equation}
Notice that since $G$ has the asymptotic behavior $G(x,p_j)\simeq -\frac{1}{2\pi}\log (\dist(x,p_j))$, then
$$
\tilde{K}(x) \simeq dist(p_j,x)^{2\alpha_j}e^{-h(x)} K(x) \quad \mbox{\quad close to $p_j$.}
$$

A possible strategy, meaningful also from the physical point of view, is to study \eqref{equation} for $\lambda$ positive independent on $\Sigma$ and $\divam$ and to deduce a posteriori the answer to the geometric question taking $\lambda=4\pi\chi(\Sigma,\divam)>0$.

Under the hypotheses $K>0$ and $\alpha_j>0$, Bartolucci and Tarantello, \cite{bt}, proved a concentration-compactness result which implies that blow-up can occur only if $\lambda$ belongs to the following discrete set of values
\beq\label{crit}
\Gamma(\underline{\alpha}_m)=\left\{8\pi r + \sum_{j=1}^{m}8\pi(1+\alpha_j)n_j\,|\,r\in \mathbb{N}\cup \{0\}, n_j\in\{0,1\}\right\}.
\eeq
Besides, they also proved an existence result for \eqref{equation} on surfaces with positive genus and $\lambda \in (8\pi,16\pi)\setminus \Gamma(\underline{\alpha}_m)$, generalized by Bartolucci, De Marchis and Malchiodi in \cite{BarDemMal}, obtaining solvability for any $\lambda\in (8\pi,+\infty) \setminus \Gamma(\underline{\alpha}_m)$.\\
 The case $K>0$ and $\alpha_j<0$ has been analyzed in \cite{Car1} and \cite{CarMal}.

Again, the problem in the sphere is more delicate. In the case $m=2$ and positive constant curvature, Troyanov, \cite{Troy2}, showed that \eqref{equation} admits solutions only if $\alpha_1=\alpha_2$ (see also \cite{BarMal}) and this implies (taking $\alpha_2=0$) that \eqref{equation} does not admit solutions for $m=1$ (see also \cite{BarLinT}). In other words, the \emph{tear drop} conical singularity on $\mathbb{S}^2$ does not admit constant curvature. Besides, for $m=2$, Chen and Li, \cite{ChenLiIneq}, gave necessary conditions on $K$ for the solvability. Eremenko in \cite{Erem} studied the case of prescribing constant positive curvature with three conical singularities. Without restrictions on the number of singularities,  Malchiodi and Ruiz, \cite{MalchiodiRuizSphere}, derived an existence result under some extra assumptions for $\lambda \in (8\pi,16\pi) \setminus \Gamma(\divam)$. We also refer the reader to \cite{Mon}, where the authors gave a criterion for the existence of a metric of constant curvature on $\mathbb{S}^2$. In a recent deep paper, \cite{ChenLin}, Chen and Lin computed the Leray-Schauder degree of \eqref{eqdelta} for $\lambda\notin \Gamma(\underline{\alpha}_m)$ recovering some of the previous existence results and deriving new ones in the case $\chi(\Sigma)>0$. Anyway on the sphere there are still different situations in which the degree is zero and the solvability is not known.

\

In this paper we consider the problem on the unit 2-sphere $\Sd$ endowed with the standard metric $g_0$, focusing on the case of $K$ sign-changing, which, as far as we know, has not been studied yet for general singular surfaces with an arbitrary number of conical singularities. This situation is admissible from the geometrical point of view, indeed, as already remarked, if $\chi(\Sd,\underline\alpha_m)>0$ Gauss-Bonnet only rules out the possibility that $K$ is non positive.

Thus from now on we will assume
\begin{enumerate}[label=(H1), ref=(H1)]
\item \label{H1} $K$ sign-changing, namely $K(x)K(y)<0$ for some $x,y \in \Sd$.
\end{enumerate}

Since problem \eqref{equation} has a variational structure, its solutions can be found as critical points of the energy functional
\beq\label{functional}
I_\lambda(u)=\frac12\int_{\Sd} |\nabla u|^2 dV_{g_0}+\frac\lambda{|\Sd|}\int_{\Sd} u \, dV_{g_0}-\lambda \log\int_{\Sd} \tilde K e^u dV_{g_0},
\eeq
defined in the domain
\beq\label{X}
X=\left\{u\in H^1(\Sd)\,|\,\int_{\Sd} \tilde K e^u\,dV_{g_0}>0\right\}.
\eeq
Notice that hypothesis \ref{H1} implies that $X$ is not empty.\\ Moreover the functional $I_\lambda$ is invariant under addition of constants, as well as problem \eqref{equation}.

Let $k\in\mathbb{N}\setminus\{0\}$, we study the case
$$
\lambda\in(8\pi k,8\pi(k+1)),
$$
for which the functional $I_\lambda$ is not bounded below and we will use a min-max scheme to find solutions of \eqref{equation}.
In this direction, we define the sets
$$
S^{\pm}=\{x\in \Sd \,|\,K(x) \gtrless 0\}, \qquad  S^0=\{x\in \Sd\,|\,K(x)=0\}, 
$$
and introduce the extra assumptions
\begin{enumerate}[label=(H2), ref=(H2)]
\item \label{H2} $K\in C^{2,\alpha}(V)$, for some neighborhood $V$ of $\partial S^+$, \\$\nabla K(x)\neq0$ for any $x\in \partial S^+$,
\end{enumerate}
\begin{enumerate}[label=(H3), ref=(H3)]
\item \label{H3} $p_j\notin \partial S^+$ \ for all $j\in\{1,\ldots,m\}$.
\end{enumerate}
Hypothesis \ref{H2} implies that the nodal line $\partial S^+ \subset S^0$ is regular, that
\beq\label{Npm}
N^{+}=\#\{\textnormal{connected components of $S^+$}\}<+\infty
\eeq 
and that $(S^0\cap \partial S^+)\setminus\partial S^-=\emptyset$, but it does not exclude that $S^0\setminus \partial S^+$ is non empty.\\
On the other hand by virtue of \ref{H3}, we can suppose, up to reordering, that there exists $\ell\in\{0,\ldots,m\}$ such that
\beq\label{ell}
p_j\in{S^+}\textrm{ for $j\in\{1\,\ldots,\ell\}$,}\quad p_j\in{S^- \cup (S^0\setminus \partial S^+)} \textrm{\; for $j\in\{\ell+1,\ldots,m\}$.} 
\eeq

As we will see the conical singularities located in $S^- \cup (S^0\setminus \partial S^+)$ do not play any role.

We state the fourth assumption on $K$
\begin{enumerate}[label=(H4), ref=(H4)]
\item \label{H4} $N^+>k$ or $S^+$ has a connected component which is non-simply connected.
\end{enumerate}

Notice that for some functions $K$ both hypotheses in \ref{H4} are fulfilled.

\

We are ready to enunciate our first existence result.
\begin{theorem}\label{thm:prescribing1}
Let $p_1,\ldots, p_m\in\Sd$, $\alpha_1,\ldots,\alpha_m>0$, $4\pi\chi(\Sd,\divam)\notin \Gamma(\diva_\ell)$, where $\ell$ is defined in \eqref{ell}. Then any function $K$ defined on $\Sd$ and satisfying \ref{H1}, \ref{H2}, \ref{H3} and \ref{H4} is the Gaussian curvature of at least one metric conformal to $g_0$ and having at $p_j$ a conical singularity with order $\alpha_j$.
\end{theorem}

The previous result is a direct consequence of the following theorem.

\begin{theorem}\label{thm:equation1}
Let $p_1,\ldots, p_m\in\Sd$, $\alpha_1,\ldots,\alpha_m>0$ and let $K$ be a function on $\Sd$ satisfying \ref{H1}, \ref{H2}, \ref{H3} and \ref{H4}. Suppose that $\lambda_0\in(8\pi,+\infty)\setminus\Gamma(\diva_\ell)$, where $\ell$ is defined in \eqref{ell}, then \eqref{equation} admits a solution for $\lambda=\lambda_0$ with $(\Sigma,g)=(\Sd,g_0)$.
\end{theorem}

In case that $k=1$, i.e. $\lambda \in (8\pi,16\pi)$, and $\alpha_1,\ldots,\alpha_{\ell} \in (0,1]$ we can describe the topology of the sublevels of $I_{\lambda}$ in a more accurate way depending on the order of the singularities located in $S^+$. In particular, we set
\beq\label{Jlambda}
J_\lambda=\{p_j\in S^+\,|\,\lambda<8\pi(1+\alpha_j)\}
\eeq
and we introduce the fifth hypothesis
\begin{enumerate}[label=(H5), ref=(H5)]
\item \label{H5} $J_{\lambda_0}\neq\emptyset$.
\end{enumerate}

\begin{theorem}\label{thm:prescribing2}
Let $p_1,\ldots, p_m\in\Sd$, $\alpha_1,\ldots,\alpha_{\ell}\in(0,1]$, $\alpha_{\ell+1},\ldots,\alpha_m>0$ and let $\lambda_0=4\pi\chi(\Sd,\divam)\in(8\pi,16\pi)\setminus\Gamma(\diva_\ell)$, where $\ell$ is defined in \eqref{ell}. Then any function $K$ defined on $\Sd$ and satisfying \ref{H1}, \ref{H2}, \ref{H3} and \ref{H5} is the Gaussian curvature of at least one metric conformal to $g_0$ and having at $p_j$ a conical singularity with order $\alpha_j$.
\end{theorem}

Theorem~\ref{thm:prescribing2} can be deduced from the following result.
\begin{theorem}\label{thm:equation2}
Let $p_1,\ldots, p_m\in\Sd$, $\alpha_1,\ldots,\alpha_{\ell}\in(0,1]$, $\alpha_{\ell+1},\ldots,\alpha_m>0$ and let $K$ be a function on $\Sd$ satisfying \ref{H1}, \ref{H2}, \ref{H3} and \ref{H5}. Suppose that $\lambda_0\in(8\pi,16\pi)\setminus\Gamma(\diva_\ell)$, where $\ell$ is defined in \eqref{ell}, then \eqref{equation} admits a solution for $\lambda=\lambda_0$ with $(\Sigma,g)=(\Sd,g_0)$.
\end{theorem}

Proofs of Theorem ~\ref{thm:equation1} and Theorem~\ref{thm:equation2} are based on a minmax argument, relying in turn on the study of the low  sublevels of the functional $I_\lambda$, in the spirit of \cite{BarDemMal,Djad,DjadliMalchiodi,MalchiodiRuizSphere}.

Once $\lambda_0\in(8\pi k,8\pi(k+1))$, for some $k\in\N$, is fixed, the strategy is to find a compact non contractible topological space $\mathcal C$ and two maps $i:\mathcal C\to\{I_\lambda\leq-L\}$, $\beta:\{I_\lambda\leq-L\}\to\mathcal C$ for $L>0$ sufficiently large, such that $\beta\circ i$ is homotopically equivalent to the identity map on $\mathcal C$.\\ This immediately implies that $i(\mathcal C)$ is non contractible.  

Then we consider the class $\mathcal G_\lambda$ of continuous maps $g$ from the topological cone $\tilde{\mathcal C}$ over $\mathcal C$ into $H^1(\Sd)$, which coincide with $i$ on the boundary of the cone, i.e. $g_{|\mathcal C}=i_{|\mathcal C}$, and we define a min-max level $\displaystyle{c_\lambda=\inf_{g\in\mathcal{G}_\lambda}\sup_{z\in\tilde{\mathcal{C}}}I_{\lambda}(g(z))}$. The noncontractibility of $\mathcal C$ allows to prove that $c_\lambda>-\infty$ and to find a Palais-Smale sequence for $I_\lambda$. In turn by the Struwe monotonicity trick we derive the existence of a solution $u_\lambda$ of \eqref{equation} for almost every $\lambda$ close to $\lambda_0$.

At this point to find a solution of \eqref{equation} with $\lambda=\lambda_0$ we need a compactness result and this yields an extra difficulty because the one by Bartolucci and Tarantello can not be applied, requiring $K$ positive. For this reason we prove an alternative compactness theorem (see Theorem \ref{thm:compactness}) to exclude that blow-up can occur if $\lambda\notin \Gamma(\diva_\ell)$. To derive this result we follow the approach employed in \cite{ChenLiDCDS} to treat the \emph{regular} Nirenberg problem, combined with an energy comparison argument. To do so, in particular to get a priori bounds for solutions in $\Sd\setminus S^+$, we assume hypotheses \ref{H2} and \ref{H3}.

For Theorem \ref{thm:equation1} we take as $\mathcal{C}$ the set of formal barycenters of order $k$ of a proper compact subset $Y$ of $S^+\setminus\{p_1,\ldots,p_m\}$, namely the family of unit measures which are supported in at most $k$ points of $S^+$. Notice that $\mathcal C$ will turn out to be non contractible in view of assumption \ref{H4}. The main underlying idea is that, if $\lambda\in(8\pi k,8\pi(k+1))$ and if $I_\lambda$ attains large negative values, the measure $\tfrac{e^u\chi_{S^+}}{\int_{S^+}e^u\,dV_g}$ has to concentrate near at most $k$ points of $S^+$. Next, we construct a global projection of $S^+$ onto $Y$ and in turn this map induces a projection from the barycenters of $S^+$ onto those of $Y$. Exactly in this way we define the map $\beta$, while for what concerns $i$ we use suitable test functions to embed $\mathcal C$ into $\{I_\lambda\leq -L\}$.

In the case of Theorem \ref{thm:equation2}, we take as $\mathcal C$ a compact subset of $S^+\setminus J_\lambda$ and to construct the map $\beta$ we study the concentration properties of the measures $\tfrac{\tilde K e^u\chi_{S^+}}{\int_{S^+}\tilde K e^u\,dV_g}$ for $u\in\{I_\lambda\leq -L\}$, adapting the arguments in \cite{MalchiodiRuizSphere}.

\begin{remark}
Theorem~\ref{thm:equation1} and Theorem~\ref{thm:equation2} can be seen as the counterparts for $K$ sign-changing of the existence results obtained in \cite{BarDemMal} and \cite{MalchiodiRuizSphere} for $K$ positive, where $S^+$ plays the role of $\Sigma$. We point out that, whereas the minmax scheme we developped to treat the sign-changing case works with some modifications on any surface endowed with any metric (see Remark \ref{remnons} and Remark \ref{remnons2}), in the proof of the compactness result we strongly use the fact that $(\Sigma,g)=(\Sd,g_0)$ in order to apply the stereographic projection.
\end{remark}
\

Finally, we prove that it is not always possible to prescribe on a singular standard sphere a sign-changing Gaussian curvature satisfying \ref{H1}, \ref{H2} and \ref{H3} when neither \ref{H4} nor \ref{H5} are fulfilled. Therefore we can say that these assumptions are somehow sharp.

\begin{theorem}\label{thm:geometricnonexistence}
Let $p\in\Sd$ and $\alpha>0$ then there exists a class of axially symmetric functions on $\Sd$ satisfying \ref{H1}, \ref{H2} and \ref{H3}, which are not the Gaussian curvature of any metric conformal to $g_0$ having at $p$ a conical singularity of order $\alpha$.
\end{theorem}

Theorem \ref{thm:geometricnonexistence} follows from next result, which is inspired by \cite{ChenLins}. Before stating it, for $p\in \Sd$, we introduce the set $\mathcal F_p \subset C^0(\Sd)$ defined as
\beq\label{setRP}
\mathcal F_p=\left\{F\in C^0(\Sd)\,:\,\begin{array}{l}\text{\small $F$ is sign-changing, rotationally symmetric}\\ \text{\small with respect to $p$, monotone in the region}\\\text{\small where it is positive and $F(-p)=\displaystyle{\max_{\Sd} F}$}\end{array}\right\}.
\eeq

\begin{theorem}\label{nonexistence}
Let $\lambda\in(8\pi,+\infty)$, $p\in \Sd$ and $\alpha>0$. Then for any function $F\in \mathcal F_p$ there exists a function $K_F$, having the same nodal regions of $F$, such that \ref{H1}, \ref{H2} and \ref{H3} are fulfilled but equation \eqref{equation} (with $m=1$, $p_1=p$, $\alpha_1=\alpha$ and $\tilde K\equiv\tilde{K}_F=e^{-h_1}K_F$) does not admit a solution.
\end{theorem}

As it will be clear from the definition of $K_F$, see \eqref{KR}, due to our assumptions on $F$, for any function $K_F$ in the statement of the previous theorem $p\in S^-$ (then $J_\lambda=\emptyset$) and $S^+$ is contractible. In particular $N^+=1$.

\

The rest of the paper is organized as follows. In Section~\ref{prenot} we fix the notation and give some preliminary results. In section~\ref{Scom} we find a priori bounds for solutions of \eqref{equation} and we prove a compactness result. In Section~\ref{low} we study the low sublevels of $I_{\lambda}$ and finally Section~\ref{sproof} is devoted
to prove Theorem~\ref{thm:equation1}, Theorem~\ref{thm:equation2} and Theorem~\ref{nonexistence}.

\

\section{Notations and preliminaries}\label{prenot}
\setcounter{equation}{0}
In this section we fix the notation used in this paper and collect some preliminary known results.

From now on $(\Sd,g_0)$ will be the unit 2-sphere equipped with the standard metric, $\dist(x,y)$ will denote the distance between two points $x, y \in \Sd$ induced by the ambient metric and $\dist(\O_1,\O_2)=\min\left\{ \dist(x,y) | x \in \Omega_1, y \in \Omega_2 \right\}$ will denote the distance between two subsets of $\Sd$. \\
Given $0<r<R$, $p\in \Sd$ and $\Omega \subset \Sd$, the symbol $B_p(r)$ stands for the open ball of radius $r$ and center $p$, $A_p(r,R)$ denotes the corresponding open annulus and $(\Omega)^r=\{x\in \Sd,|\,\dist(x,\Omega)<r\}$. \\
Let $f \in L^1(\Sd)$, we set $\dashint_{\Sd} f=\frac1{|\Sd|} \int_{\Sd} f$, where $|\Sd|$ is the area of $\Sd$. We will use the same notation for a subset $T$ of $\R^2$, namely $(T)^r=\{x\in \R^2\,|\,\dist(x,T)<r\}$.\\
For a real number $a$, we introduce the following notation for the sublevels of the energy functional
$$
I_{\lambda}^{a}= \{u\in X\,:\,I_\lambda(u)\leq a\},
$$
where $I_{\lambda}$ and $X$ are defined in \eqref{functional} and \eqref{X} respectively. \\
For any $A \subset \Sd$, $\overline{A}$ will denote its closure, $\interior(A)$ its interior and $\chi_A(x)$ the characteristic function of $A$. Moreover, for $A, B\subset \Sd$, $A\vartriangle B$ stands for their symmetric difference. 

Given a metric space $M$ and $k\in \mathbb{N}$, we denote by $M_k$ the set of formal barycenters of order $k$ on $M$, namely the following family of unit measures supported in at most $k$ points
\beq\label{formbar}
M_k=\left\{\sum_{i=1}^k t_i \delta_{x_i}, \sum_{i=1}^k t_i =1 , x_i \in M \right\}.
\eeq

Positive constants are denoted by $C$, and the value of $C$ is allowed to vary from formula to formula.

\

\subsection{Moser-Trudinger inequalities}
\

Moser-Trudinger type inequalities are a powerful tool in our study because they allow to deduce properties of the functional $I_\lambda$ defined in \eqref{functional}. We start recalling a weaker version of the classical Moser-Trudinger inequality, see \cite{Mos1}.
\begin{proposition}\label{prop:MT}
Let $\Sigma$ be a compact surface, then there exists $C>0$ such that
\beq
\log\int_{\Sigma}e^{u}dV_g \leq\frac{1}{16\pi} \int_\Sigma |\nabla u|^2\ dV_g\,+\frac1{|\Sigma|}\int_\Sigma u dV_g+C, \qquad\forall\,u\in H^1(\Sigma).
\eeq
\end{proposition}

As an easy application of the previous proposition, we have
$$ I_{\lambda}(u) \geq \frac{8\pi -\lambda}{16\pi}\int_{\Sd} |\nabla u|^2\ dV_g - C,$$
for all $u \in X$. In particular, $I_{\lambda}$ is coercive for $\lambda \in \left(0, 8\pi \right)$, and a solution for \eqref{equation} can be found as a minimizer.

For larger values of the parameter $\lambda$ the previous inequality does not give any information. In particular, for $\lambda>8\pi$ it can be easily seen that the functional is not bounded from below. See Lemma~\ref{notlimbl2} and Lemma~\ref{notlimb}.

Subsequently, we recall a result which roughly speaking states that if, into $\ell+1$ regions of a surface $\Sigma$, $e^u$ has integral controlled from below (in terms of $\int_\Sigma e^u\,dV_g$), the constant $\frac1{16\pi}$ can be basically divided by $\ell+1$. The following proposition has been proved for the first time in \cite{ChenLiIneq}, with $\tilde H=1$ and $\ell=1$, and generalized in \cite{Djad} for $\ell>1$. Assuming $\tilde H$ only bounded does not require any changes in the arguments of the proof.

\begin{proposition}\label{prop:ChenLiIneq}
Let $\Sigma$ be a compact surface, $\tilde H:\Sigma\to\R$, with $0\leq\tilde H(x)\leq C_0$. Let $\ell$ a positive integer and $\Omega_1, \ldots, \Omega_{\ell+1}$ be subsets of $\Sigma$ with $\dist(\Omega_i, \Omega_j)\geq\delta$, for $i\neq j$, where $\delta$ is a positive real number, and fix $\gamma\in(0,\frac{1}{\ell+1})$.

Then for any $\eps>0$ there exists a constant $C=C(C_0,\eps,\delta,\gamma)$ such that
\beq\label{MTimproved}
\log\int_{\Sigma}\tilde H e^{u}dV_g\leq\frac{1}{16(\ell+1)\pi-\eps}\int_\Sigma |\nabla u|^2 dV_g\,+\frac1{|\Sigma|}\int_\Sigma u\,dV_g +C
\eeq
for all functions $u\in H^1(\Sigma)$ satisfying
\beq\label{measspread}
\frac{\int_{\Omega_i}\tilde He^u\,dV_g}{\int_\Sigma \tilde H e^u\,dV_g}\geq\gamma, \qquad\textnormal{for $i=1,\ldots,\ell+1$.}
\eeq
\end{proposition}

Next we recall a criterion which gives sufficient conditions for \eqref{measspread} to hold. We refer to \cite{DjadliMalchiodi} for the proof.
\begin{lemma}\label{lemma:criterion}
Let $\ell$ be a positive integer and suppose that $\eps$ and $r$ are positive numbers and that for a non-negative function $f\in L^1(\Sigma)$ with $\|f\|_{L^1(\Sigma)}=1$ there holds
$$
\int_{\cup_{i=1}^{\ell} B_{p_i}(r)}f dV_g<1-\eps,\qquad\textnormal{for any $\ell$-tuple $p_1,\ldots,p_l \in \Sigma$}.
$$
Then there exist $\bar\eps>0$ and $\bar r>0$ depending only on $\eps$, $r$ and $\Sigma$ (but not on $f$), and $\ell+1$ points $\bar p_1,\ldots, \bar p_{\ell+1}\in \Sigma$ (which depend on $f$) satisfying
$$
\int_{B_{\bar  p_i}(\bar r)}f dV_g\geq\bar\eps, \qquad B_{\bar p_i}(2 \bar r)\cap B_{\bar p_j}(2\bar r)=\emptyset \quad \mbox{for $i,j=1,\ldots,\ell+1$ and $i\neq j$}.
$$
\end{lemma}

Next we introduce a localized version of the Moser-Trudinger inequality obtained in \cite{MalchiodiRuizSphere,DjadliMalchiodi}.

\begin{proposition}\label{prop:MTlocalized}
Assume that $\Sigma$ is a compact surface (with or without boundary), and $\tilde H:\Sigma\to\R$ measurable, $0\leq \tilde H(x)\leq C_0$ a.e. $x\in \Sigma$. Let $\Omega\subset \Sigma$, $\delta>0$ such that $\dist(\Omega,\partial \Sigma)>\delta$.

Then, for any $\eps>0$ there exists a constant $C=C(C_0,\eps,\delta)$ such that for all $u\in H^1(\Sigma)$,
\beq\label{eq:MTineqWeak}
\log \int_{\O} \tilde H(x) e^u\,dV_g\leq \frac 1{16\pi-\eps}\int_\Sigma |\nabla_g u|^2\,dV_g+\dashint_\Sigma u\,dV_g+C.
\eeq
\end{proposition}

\

\subsection{A priori estimates on the entire solutions}\label{subapr}
\

This subsection is devoted to present some a priori $L^{\infty}$ bounds for solutions of the problem
\beq\label{CLp}
-\Delta u = R(x) e^u \quad \mbox{in $\mathbb{R}^2$},
\eeq
where $R$ is a sign-changing function.\\
The following results are originally due to Chen and Li (see \cite{ChenLiDCDS,ChenLiAnn}) who proved them in order to derive a priori bounds for solutions of the \emph{regular} Nirenberg problem. We will focus on the estimates in the region where $R\leq\eps$, for some small $\eps$. They performed a stereographic projection to transform the equation on $\Sd$ into \eqref{CLp}. In particular, in their case, it was natural to assume the following asymptotic growth of the solutions at infinity
\beq\label{agN}
u \sim -4 \log |x|
\eeq
and that $R$ has a limit as $|x|\to+\infty$, for example,
\beq\label{SH}
\lim_{|x|\to+\infty}R(x)\in(0,+\infty).
\eeq

Actually, we will show that their approach, with proper modifications (see Lemma~\ref{lem21} below), allows to deal also with solutions of \eqref{CLp} behaving at infinity as
\beq\label{alphabeh}
u \sim -\eta \log |x|,
\eeq
for some $\eta > 4$, if the function $R$ satisfies
\beq\label{coerciveta}
\lim_{|x|\to+\infty}R(x)|x|^{4-\eta}\in(0,+\infty).
\eeq

Chen and Li in \cite{ChenLiDCDS} assumed $R\in C^{2,\alpha}(\R^2)$ and $\nabla R(x)\neq0$ in $\{x\in\R^2\,|\,R(x)=0\}$, but in fact their proof required only
\begin{equation}\label{hpCL}
R^0=\overline{R^+}\cap \overline{R^-},
\end{equation}
\begin{equation}\label{hpregularity}
R\in C^0(\R^2)\cap C^{2,\alpha}(U)\;\:\textnormal{and $\nabla R(x)\neq0$ for $x\in\overline{R^+}\cap \overline{R^-}$},
\end{equation}
where
$$
R^0=\{x\in\R^2\,|\,R(x)=0\},\quad R^\pm=\{x\in\R^2\,|\,R(x)\gtrless0\},
$$
and $U$ is a neighborhood of $\overline{R^+}\cap\overline{R^-}$.\\
In general, if we assume \eqref{hpregularity} and not  \eqref{hpCL}, $R^0$ can be decomposed in the following disjoint union
\[
R^0=(\overline{R^+}\cap\overline{R^-})\amalg \mathcal{Q}^+\amalg \mathcal{Q}^-
\]
where $\mathcal{Q}^\pm$ are such that $\partial \mathcal{Q}^\pm\subset \overline{R^\pm}\setminus\overline{R^\mp}$, i.e. $\mathcal{Q}^\pm$ are the (possibly empty) components of $R^0$ surrounded by positive/negative nodal regions of $R$.\\
Let us set
\[
\mathcal{Q}=\mathcal{Q}^+\cup\mathcal{Q}^-.
\]
At last we define
\begin{equation}\label{erre}
r=\frac13\min\{\dist(\mathcal{Q}^+, \overline{R^-}),\dist(\mathcal{Q}^-, \overline{R^+})\}>0
\end{equation}
and state a generalized version of Proposition 2.1 and Proposition 3.1 in \cite{ChenLiDCDS}, replacing assumptions \eqref{agN} and \eqref{SH} with \eqref{alphabeh} and \eqref{coerciveta} respectively and removing hypothesis \eqref{hpCL}.

\begin{theorem}\label{thm:ChenLi}
Assume that $R$ verifies \eqref{coerciveta} and \eqref{hpregularity} and that there exist $\beta, \delta>0$ such that $|\nabla R(x) | \geq \beta$ for any $x\in\{x\in\R^2\,|\,|R(x)|\leq \delta\}\setminus(\mathcal{Q})^{r}$. Then there are positive constants $\varepsilon$ and $C$, depending only on $\beta,\delta$, $\|R\|_{ C^{2,\alpha}(U)}$ and $\min_{\R^2}R$, such that for any solution u of \eqref{CLp} satisfying \eqref{alphabeh}, $u\leq C$ in  $\{x \in \mathbb{R}^2 | R(x)\leq\varepsilon \} \setminus (\mathcal{Q}^+)^{r}$.
\end{theorem}

\begin{remark}\label{rem:Rflambda}
Let us consider the following family of positive functions
\[
F_\eta(x)=2\pi\eta\left(\frac{4}{(1+|x|^2)^2}\right)^{1-\frac{\eta}{4}},
\]
depending on a parameter $\eta$ varying in a bounded subset $I\subset (4,+\infty)$.\\
If we assume that $R$ verifies \eqref{hpregularity}, that there exist $\beta,\,\delta>0$ such that $|\nabla R(x)|\geq \beta$ for any $x\in \{x\in \R^2\,|\,|R(x)|\leq\delta\}\setminus(\mathcal Q)^r$ and that  $\lim_{|x|\to+\infty}R(x) F_\eta(x)|x|^{4-\eta}\in(0,+\infty)$, then it can be seen that, for any solution $u$ of $-\Delta u =R F_\eta e^u$ in $\R^2$ satisfying \eqref{alphabeh}, $u\leq C$ in $\{x \in \mathbb{R}^2 | R(x)\leq\varepsilon \} \setminus (\mathcal{Q}^+)^{r}$, where $C$ and $\varepsilon$ are positive constants depending on $\beta$, $\delta$, $L=\sup_{\eta\in I}\|F_\eta\|_{C^{2,\alpha}(U)}$ and $M=\inf_{\eta\in I}\min_{x\in\R^2}R(x)F_\eta(x)$ (with $L<+\infty$ and $M>-\infty$) but not on $\eta$.
\end{remark}

Theorem~\ref{thm:ChenLi} follows from Proposition~\ref{propChLi1} and Lemma~\ref{lem28}.

\begin{proposition}\label{propChLi1}
Assume that $R$ verifies \eqref{coerciveta} and \eqref{hpregularity} and that there exist $\beta, \delta>0$ such that $|\nabla R(x) | \geq \beta$ for any $x\in\{x\in\R^2\,|\,|R(x)|\leq \delta\}\setminus(\mathcal{Q})^{r}$. Then there are positive constants $\varepsilon$ and $C$, depending only on $\beta,\delta$, $\|R\|_{ C^{2,\alpha}(U)}$ and $\min_{\R^2}R$, such that for any solution u of \eqref{CLp} satisfying \eqref{alphabeh}, $u\leq C$ in  $\{x \in \mathbb{R}^2 | R(x)\leq\varepsilon \} \setminus (\mathcal{Q})^{r}$.
\end{proposition}

The proof of the previous proposition can be recovered mimicking the ideas of Proposition~3.1 in \cite{ChenLiDCDS}, once Lemma~2.1 in \cite{ChenLiDCDS} is substituted by Lemma~\ref{lem21} below. Indeed, one can prove that $\int R e^u dx$  is bounded in any ball where $R$ is strictly negative and combining this fact with Lemma~\ref{lem21}, one gets that a solution $u$ of \eqref{CLp} verifying \eqref{alphabeh} is bounded from above in the region $\{x \in \mathbb{R}^2 | R(x)\leq-\eps \} \setminus (\mathcal{Q})^{r}$. Next, by virtue of assumption \eqref{hpregularity}, one can extend the estimate in the whole region $ \{x \in \mathbb{R}^2 | R(x)<\varepsilon \} \setminus (\mathcal{Q})^{r}$ via a local moving plane method using the regularity of $R$ in $U$ and an estimate of $\int R e^u dx$ in a ball where $R$ is strictly positive.

\begin{lemma}\label{lem21}
Let $x_0$ be such that $R(x_0)<0$. Let $3\eps_0=\dist\left( x_0,R^0\right)$ and let
$$
R(x)\leq -\delta \quad \forall x \in B_{x_0}(2\eps_0),
$$
for a fixed $\delta>0$. Moreover, assume that $R$ satisfies \eqref{coerciveta}.
Then, for every solution $u$ of \eqref{CLp}, verifying \eqref{alphabeh},
\beq\label{conlem21}
u(x_0) \leq u(x) + \eta \log\left(\frac{|x-x_0|}{\eps_0} +1 \right) + C, \qquad \forall\, x \in \mathbb{R}^2,
\eeq
where $C$ is a constant depending only on $\min_{\R^2}R$ and $\delta$.
\end{lemma}

\begin{proof}
Let $\hat{x} \in B_{x_0}(\varepsilon_0)$, we claim that
\beq\label{eq1lem21}
u(x) \leq u\left( \hat{x}+\varepsilon_0^2\frac{x-\hat{x}}{|x-\hat{x}|^2}\right)-\eta \log \frac{|x-\hat{x}|}{\varepsilon_0}+C, \quad \forall x \in B_{\hat{x}}(\varepsilon_0).
\eeq

The point $\hat{x}+\varepsilon_0^2\frac{x-\hat{x}}{|x-\hat{x}|^2}$ corresponds to the reflection of $x$ about $\partial B_{\hat{x}}(\varepsilon_0)$. Let $v(x)=u\left(\hat{x}+\varepsilon_0^2\frac{x-\hat{x}}{|x-\hat{x}|^2}\right)-\eta\log \frac{|x-\hat{x}|}{\varepsilon_0}$, which satisfies:
$$
-\Delta v(x)=R\left(\hat{x}+\varepsilon_0^2\frac{x-\hat{x}}{|x-\hat{x}|^2}\right)\left(\frac{|x-\hat{x}|}{\varepsilon_0}\right)^{\eta-4}e^{v(x)}.
$$

Next, we take the auxiliary function $w(x)=v(x)-u(x)+\gamma$, where $\gamma$ is a positive parameter to be determined. Then for $x\in B_{\hat{x}}(\varepsilon_0)$ 
\begin{align*}
& \Delta w + R(x)e^{\phi(x)}w(x) \\
& =\left[ R(x)e^{\gamma}- R \left( \hat{x}+\varepsilon_0^2\frac{(x-\hat{x})}{|x-\hat{x}|^2} \right) \left(\frac{|x-\hat{x}|}{\varepsilon_0}\right)^{\eta-4}\right]e^{v(x)} \leq \left(-\delta e^{\gamma}-m\right)e^{v(x)}
\end{align*}
where $m=\min_{\R^2}R$ and $\phi$ is a function which is between $\gamma + v(x)$ and $u(x)$. Next, we choose $\gamma$ large enough such that $-\delta e^{\gamma} - m \leq 0$. Since $w(x)=\gamma$ in $\partial B_{\hat{x}}(\varepsilon_0)$, by the maximum principle we obtain that
$$
w(x)\geq 0 \quad \mbox{ in $B_{\hat{x}}(\varepsilon_0)$},
$$
which implies \eqref{eq1lem21}. \\
To end the proof, observe that for every $x \in \mathbb{R}^2$ there exists $\hat{x}$ in $ B_{x_0}(\varepsilon_0)$ such that the reflection of $x_0$ about $\partial B_{\hat{x}}(\varepsilon_0)$ is the point $x$, i.e. for all $x \in \mathbb{R}^2$ there exists $\hat{x} \in B_{x_0}(\varepsilon_0)$ such that $x = \hat{x} + \varepsilon_0^2\frac{x_0-\hat{x}}{|x_0-\hat{x}|^2}$. Clearly, $|x_0-\hat{x}||x-\hat{x}|=\varepsilon_0^2$, so \eqref{eq1lem21} implies directly \eqref{conlem21}.
\end{proof}

Next lemma allows to extend the a priori bound to the region $(\mathcal{Q}^-)^r$.

\begin{lemma}\label{lem28}
Under the assumptions of Proposition~\ref{propChLi1}, then $u(x)\leq C$ for any $x \in (\mathcal{Q}^-)^{r}$.
\end{lemma}
\begin{proof}
Let us consider an open regular subset $\Omega$ of $R^-$ such that $\overline\Omega\subset R^-$ and $\Omega\supset (\mathcal{Q}^-)^{2r}$. Proposition~\ref{propChLi1} asserts that $u(x)\leq C$ for any $x\in \partial\Omega$. Then, by our assumptions we have that $-\Delta u \leq 0$ in $\Omega$ and $u\leq C$ in $\partial \Omega$. Applying the weak maximum principle, we reach the desired conclusion.
\end{proof}

\subsection{Non existence result for the \emph{regular} Nirenberg problem}
\

We recall a non existence result for the problem \eqref{CLp}, obtained in \cite{ChenLins} to prove that the \emph{regular} Nirenberg problem \eqref{Nirenberg} does not admit solution for $K$ axially symmetric, sign-changing and monotone in the region where $K$ is positive.

Let $r_0>0$ and $R \in C^{0}_{rad}(\mathbb{R}^2)$ such that
\beq\label{NE}
\mbox{$R$ is positive and non-increasing for $r<r_0$, negative for $r>r_0$}.
\eeq

The following result has been derived under the hypothesis $\eqref{agN}$, however, as shown below, it holds under the less restrictive assumption \eqref{alphabeh}.
\begin{theorem}\label{nethm:ChenLi}
Assume that $R\in C^{0}_{rad}(\mathbb{R}^2)$ is a bounded function verifying \eqref{NE}. Then there is no solution for the problem $\eqref{CLp}$ such that $\eqref{alphabeh}$ holds.
\end{theorem}

The key point to derive this generalized result is to modify properly Lemma~2.1 in \cite{ChenLins}, taking into account the new asymptotic behavior.

\begin{lemma}\label{lemnon}
Let $R\in C^0_{rad}(\R^2)$ be a bounded function such that
\beq\label{lemnoneq}
R(r)>0, \,  R'(r)\leq 0 \, \mbox{ for } r<1; \quad R(r)\leq 0 \mbox{ for } r\geq 1.
\eeq
Let $u$ be a solution of \eqref{CLp} such that \eqref{alphabeh} holds, then
\beq\label{tlemnon}
u(\mu x) > u\left(\frac{\mu x}{|x|^2}\right)-\eta\log|x| \quad \forall x \in B_0(1), \quad 0<\mu\leq 1.
\eeq
\end{lemma}
\begin{proof}
\
\textbf{Step 1:} We claim that \eqref{tlemnon} is true for $\mu=1$.\\
Let $v(x)=u\left(\frac{ x}{|x|^2}\right)-\eta\log|x|$, then $v$ verifies
$$
-\Delta v=|x|^{\eta-4}R\left(\frac{1}{|x|}\right)e^v.
$$
By \eqref{lemnoneq}, $\Delta u < 0$ and $\Delta v \geq 0$ in $B_0(1)$, then $-\Delta(u-v)>0$ in $B_0(1)$. Since $u=v$ in $\partial B_0(1)$, then
$$
u>v \quad \mbox{ in } B_0(1)
$$
by using the maximum principle.

\textbf{Step 2:}
At this point, we move  $\partial B_{0}(\mu)$ towards $\mu=0$. Let $u_{\mu}(x)=u(\mu x) + 2\log \mu$ and $v_{\mu}(x)=u_{\mu}\left(\frac{ x}{|x|^2}\right)-\eta\log|x|$, then
$$
-\Delta u_{\mu} = R(\mu |x|) e^{u_{\mu}}, \quad -\Delta v_{\mu} = |x|^{\eta-4}R\left(\frac{\mu}{|x|}\right) e^{v_{\mu}}.
$$

Taking the auxiliary function $w_\mu=u_\mu-v_\mu$, we obtain that
\beq\label{tlemnon2}
\Delta w_{\mu} + |x|^{\eta-4}  R\left(\frac{\mu}{|x|}\right)e^{\phi_{\mu}(x)}w_{\mu}(x) =\left[ R\left(\frac{\mu}{|x|}\right)|x|^{\eta-4}- R (\mu |x|)\right] e^{u_{\mu}(x)}
\eeq
for $x\in B_{0}(1)$ where $\phi_\mu$ is a function between $u_{\mu}(x)$ and $v_{\mu}(x)$. Observe that by \eqref{lemnoneq}, we have that
\[
R\left(\frac{\mu}{|x|}\right)|x|^{\eta-4}-R(\mu|x|)\leq0,\quad \textnormal{ for   $|x|\leq1$  and  $\mu\leq1$.}
\]
Therefore
\begin{equation}\label{disuguaglianza}
\Delta w_\mu+C_\mu (x)w_\mu\leq0,
\end{equation}
where $C_\mu(x)$ is a bounded function if $\mu$ is bounded away from $0$, moreover for any $\mu$ strict inequality for \eqref{disuguaglianza} holds somewhere. Thus, applying the strong maximum principle, to get \eqref{tlemnon} it is enough to show that
\beq\label{tlemnon3}
w_{\mu}(x)\geq 0 \quad \mbox{ in $B_{0}(1)$}.
\eeq
From Step 1 \eqref{tlemnon3} is true for $\mu=1$.
Next, we decrease $\mu$. By contradiction, suppose that there exists $\mu_0>0$ such that \eqref{tlemnon3} is true for $\mu\geq \mu_0$ and fails for $\mu<\mu_0$. For $\mu=\mu_0$ we can use the strong maximum principle and then the Hopf lemma in \eqref{tlemnon2} to obtain that
$$
w_{\mu_0}>0 \; \mbox{ in } B_0(1) \quad \mbox{ and } \quad \frac{\partial w_{\mu_0}}{\partial r } <0 \; \mbox{ on } \partial B_0(1).
$$

In addition, by the minimality of $\mu_0$ for any sequence $\mu_n\nearrow\mu_0$ there exists $x_n \in B_0(1)$  verifying $w_{\mu_n}(x_n)<0$. This, combined with the fact that $w_{\mu_n}=0$ on $\partial B_0(1)$, implies that there exists some $y_n$ on the segment connecting $x_n$ and $\frac{x_n}{|x_n|}$ so that $\frac{\partial w_{\mu_n}}{\partial r } (y_n) >0$. Up to a subsequence $x_n\to x_0 \in \overline{B_0(1)}$ with $w_{\mu_0}(x_0)\leq 0$, so $x_0 \in \partial B_0(1)$ and $y_n\to x_0$. Thus $\frac{\partial w_{\mu_0}}{\partial r }(x_0) \geq 0$ and we get the desired contradiction. Therefore \eqref{tlemnon3} holds for any $\mu\in(0,1]$.
\end{proof}

\begin{proof}[Proof of Theorem~\ref{nethm:ChenLi}]
By virtue of \eqref{NE} we have that \eqref{lemnoneq} holds, so applying Lemma~\ref{lemnon} we get \eqref{tlemnon}. Letting $\mu \to 0$ in \eqref{tlemnon} we obtain that $\log |x| > 0 $ for $|x|<1$ which is a contradiction.
\end{proof}

\section{An estimate in the region where $K$ is small and a compactness result}\label{Scom}
\setcounter{equation}{0}

The main result of this section is the following compactness theorem for solutions of \eqref{equation} with $(\Sigma,g)=(\Sd,g_0)$.

\begin{theorem}\label{thm:compactness}
Let $p_1,\ldots,p_m\in\Sd$, $\alpha_1,\ldots,\alpha_m>0$ and let $K$ be a Lipschitz function on $\Sd$ satisfying \ref{H1}, \ref{H2} and \ref{H3}. Let $k\in \mathbb{N}$,
$$
\displaystyle{\lambda_0\in\left(8\pi k,8\pi(k+1)\right)\setminus\Gamma(\diva_\ell)},
$$
$\lambda_n\to\lambda_0$ and $u_n$ a sequence in $X$ of solutions of \eqref{equation} with $(\Sigma,g)=(\Sd,g_0)$ and $\lambda=\lambda_n$. Assume that $I_{\lambda_n}(u_n)$ is bounded from above. Then, up to a subsequence, $u_n-\log\int_{\Sd}\tilde K e^{u_n}dV_{g_0}\to u_0$ strongly in $C^2(\Sd)$, where $u_0$ is a solution of \eqref{equation} with $(\Sigma,g)=(\Sd,g_0)$ and $\lambda=\lambda_0$.
\end{theorem}

It is worth to point out that the concentration compactness theorem due to Bartolucci-Tarantello \cite{bt}, applied for instance in \cite{BarDemMal,BarMal,MalchiodiRuizSphere}, is useless here because it requires $K$ to be non negative. Indeed, we have also to rule out the possibility that some blow up occurs in $S^0\cup S^-$, where we keep the notations introduced in Section \ref{Introduction}:
$$
S^{\pm}=\{x\in \Sd\,|\,K(x) \gtrless 0\}, \quad S^0=\{x\in \Sd\,|\,K(x)=0\}.
$$

To do so, we take profit of Theorem~\ref{thm:ChenLi} in order to obtain a priori bounds in a neighborhood of $S^0\cup S^-$. In particular, via a stereographic projection, we can prove the following.

\begin{theorem}\label{thm:ChenLinostro}
Assume that $K$ is a Lipschitz function on $\Sd$ satisfying \ref{H1}, \ref{H2} and \ref{H3}. Let $k\in \mathbb{N}$, then there exist $\eps$, $C>0$ (depending only on $\beta$, $\delta$, $\|K\|_{C^{2,\alpha}(V)}$), such that, for any $u$ solution of the equation \eqref{equation} with $(\Sigma,g)=(\Sd,g_0)$ and $\lambda\in(8\pi k,8\pi(k+1))$, satisfying $\int_{\Sd} \tilde K e^u dV_{g_0}=1$, then $ u\leq C$ in the region where $K\leq \eps$.\\
Notice that $C$ and $\eps$ do not depend on $\lambda$ nor on $u$.
\end{theorem}

\begin{remark}
Since the equation \eqref{equation} is invariant under addition of constants, it is immediate to see that if $\int_{\Sd} \tilde K e^u dV_{g_0}$ is not fixed, the conclusion of the above theorem fails.
\end{remark}

\begin{proof}[Proof of Theorem~\ref{thm:ChenLinostro}]

Without loss of generality, suppose that $q_1=(0,0,1) \in S^+\setminus\{p_1,\ldots,p_\ell\}$. Let $P$ be the stereographic projection from $\Sd \setminus \{ q_1 \}$ to $\mathbb{R}^2$ defined by
\beq\label{SteoPro}
P(x_1,x_2,x_3)=(y_1,y_2), \quad y_i=\dfrac{x_i}{1-x_3}, \quad i=1,2.
\eeq

The inverse map $P^{-1}: \mathbb{R}^2 \mapsto \Sd \setminus\{q_1\}$ is
\beq\label{ISteoPro}
P^{-1}(y_1,y_2)=\dfrac{1}{1+\vert y \vert^2} (2y_1,2y_2,\vert y \vert^2-1).
\eeq

For any function $\psi$ on $\Sd$
$$\int_{\Sd} \psi(x) dV_g=\int_{\mathbb{R}^2}\psi(P^{-1}(y)) \dfrac{4}{(1+\vert y \vert^2)^2}dy.$$

Let $u$ be a solution of \eqref{equation}, we introduce the following variable change
\beq\label{cvar}
v(y)=u(P^{-1}(y))+\dfrac{\lambda}{8\pi}\log \left( \dfrac{4}{(1+\vert y \vert^2)^2} \right),
\eeq
then $v$ verifies
\beq\label{eqv}
-\Delta v  = \tilde{K}(P^{-1}(y))f_\lambda(y) e^{v}  \qquad \text{in $\mathbb{R}^2$,}
\eeq
with asymptotic growth at infinity
\beq\label{ieqv}
v \sim -\frac{\lambda}{2\pi}\log |y|,
\eeq
where
\beq\label{eqv1}
f_{\lambda}(y)=\lambda\left(\dfrac{4}{(1+|y|^2)^2}\right)^{1-\frac{\lambda}{8\pi}}.
\eeq

Let us set $R(y)=\tilde K(P^{-1}(y))$, moreover in the notations of Remark \ref{rem:Rflambda} $F_\eta=f_\lambda$ for $\eta=\frac{\lambda}{2\pi}\in(4k,4(k+1))$.\\ Let us first notice that the assumptions \ref{H1}, \ref{H2}, \ref{H3} on $K$ and \eqref{tildeK} guarantee that $R$ satisfies \eqref{hpregularity}. Besides, by our choice of $q_1$:
\[
\lim_{|y|\to+\infty} R(y) F_\eta(y) |y|^{4-\eta}\in(0,+\infty).
\]
Furthermore, by \ref{H2} and \ref{H3}, $\nabla R\neq0$ in $P(\partial S^+)$, then there exist $U'\subset \R^2$ and $\beta>0$ such that $P(\partial S^+)\subset U'\subset U$ and $|\nabla R|\geq\beta$ in $U'$.\\
Next, for $r$ defined in \eqref{erre} there exist $\delta>0$ such that if $|R(y)|\leq\delta$ then either $y\in U'$ or $y\in (\mathcal Q)^r$, hence clearly
\[
|\nabla R(y)|\geq\beta\qquad\mbox{for any $y\in\{y\in\R^2\mid |R(y)|\leq\delta\}\setminus (\mathcal Q)^r$.}
\]
Then by Remark \ref{rem:Rflambda} we have that there exist $\tilde C$ and $\tilde{\varepsilon}$ independent on $\lambda$ and $v$ such that
\[
v\leq\tilde C\qquad\mbox{in $\{y\in\R^2\mid R(y)\leq\tilde{\varepsilon}\}\setminus (\mathcal Q^+)^r$.}
\]
This in turn implies, by definition of $\tilde K$ and being $\mathcal{Q}^+=P(\{p_1,\ldots,p_\ell\})$, that there exist $\varepsilon>0$ independent on $\lambda$ and $v$ such that
\beq\label{vbounds}
v\leq \tilde C\qquad\mbox{in $\{y\in\R^2\mid K(P(y))\leq\varepsilon\}$.}
\eeq
At last in order to deduce the thesis we fix a point $q_2\neq q_1$ such that $q_2\in S^+\setminus\{p_1,\ldots,p_\ell\}$. 
It is immediate to see that \eqref{vbounds} combined with \eqref{cvar} gives that there exists $C>0$ such that
\[
u\leq C\qquad\mbox{in $\{x\in \Sd\mid K(x)\leq\varepsilon\}\setminus B_{q_1}(\frac13\dist(q_1,q_2))$.}
\]
To obtain the estimate in $B_{q_1}(\tfrac13\dist(q_1,q_2))$, eventually with a larger $C$ and a smaller $\eps$, it is enough to repeat the arguments replacing $q_2$ to $q_1$.

\end{proof}

\begin{proof}[Proof of Theorem~\ref{thm:compactness}]
Let us assume without loss of generality that $\int_{\Sd} \tilde K e^{u_n} dV_{g_0} =1$. If $u_n$ is bounded, up to a subsequence, $u_n\rightharpoonup u_0$. Standard elliptic arguments show that the convergence is strong and that $u_0$ is the required solution.\\
 Assume now that $\|u_n\|_{H^1(\Sd)}\to+\infty$, as $n\to+\infty$. By Theorem~\ref{thm:ChenLinostro}, there exist $\varepsilon, C>0$ such that $u_n\leq C$ in
$$
S_{\eps}=\{x \in \Sd \mid K(x)\leq \eps \}.
$$
Eventually taking a smaller $\varepsilon$, we can assume that $\{p_1,\ldots,p_\ell\}\cap S_\varepsilon=\emptyset$.

Let us set
$$
\tilde{H}(x)=\left\{\begin{array}{ll}
\tilde{K}(x)  \qquad & \text{in $S^+\setminus S_{\eps}$,}\\
\eps \qquad  & \text{in $S_{\eps}$.}
\end{array}\right.
$$

Observe that $\tilde{H}\geq 0$ and  $\tilde{H}(y)=0$ if only if $y \in \{p_1,\ldots,p_\ell\}$. Moreover, for our choice of $\varepsilon$, there exists $M\geq0$ such that $h_m(x)\geq-M$ for any $x\in S_{\eps}$ and so $\tilde{K}\leq e^M\,\tilde{H}$.

Now, inspired by Proposition~2.5 of \cite{RuizSoriano}, we define the comparison functional
$$
E_\lambda(u)=\frac12\int_{\Sd}|\nabla u|^2 dV_{g_0} +\frac\lambda{|\Sd|}\int_{\Sd} u\, dV_{g_0}-\lambda\log\int_{\Sd} \tilde{H}(x) e^u dV_{g_0}\qquad\textnormal{in $H^1(\Sd)$}.
$$
It is immediate to verify that
\beq\label{star}
I_{\lambda_n}(u_n) \geq E_{\lambda_n}(u_n)-\lambda_n\,M.
\eeq

We introduce the unit measures
$$
\mu_n=\frac{\tilde{H}e^{u_n}}{\int_{\Sd} \tilde{H}e^{u_n} dV_{g_0}},
$$
then, up to a subsequence, $\mu_n\rightharpoonup\mu$ in the sense of weak convergence of measures.

\

\textbf{Step 1:} $\mu=\sum_{i=1}^k t_i\delta_{q_i}$, for $k$ points ${q_i}\in \Sd$, $t_i \in \left[0,1\right]$ and $\sum_{i=1}^k t_i=1$.

\

Roughly speaking, the idea is that either $\mu_n$ concentrates near at most $k$ points or $\|u_n\|_{H^1(\Sd)}\leq C$, which would contradict our assumption.
To prove the claim we just need to show that there exist $k$ points $q_i$ in $\Sd$ such that for any $\gamma>0$, $r>0$ and for $n$ sufficiently large

\beq\label{*}
\frac{\int_{\Sd\setminus \cup_{i=1}^k B_{q_i}(r)}\tilde{H}e^{u_n} dV_{g_0}}{\int_{\Sd} \tilde{H}e^{u_n} dV_{g_0}}<\gamma \quad \mbox{up to a subsequence.}
\eeq

We suppose by contradiction that there exist $\gamma>0$, $r>0$ such that for any $k-$tuple $q_1,\dots,q_k \in \Sd$
$$
\frac{\int_{\Sd\setminus \cup_{i=1}^k B_{q_i}(r)}\tilde{H}e^{u_n} dV_{g_0}}{\int_{\Sd} \tilde{H} e^{u_n}dV_{g_0}}>\gamma.
$$
Then by Lemma \ref{lemma:criterion} there exist $\bar\gamma>0$ and $\bar r>0$, depending only on $\gamma$, $r$ and $\Sd$ (but not on $n$) and $k+1$ points $\bar q_{i,n}$, depending on $n$, satisfying
$$
\frac{\int_{B_{\bar{q}_{i,n}}(\bar r)}\tilde{H}e^{u_n}dV_{g_0}}{\int_{\Sd} \tilde{H}e^{u_n}dV_{g_0}}\geq\bar\gamma,\quad B_{\bar q_{i,n}}(2\bar r)\cap B_{\bar q_{j,n}}(2\bar r)=\emptyset, \mbox{ for $i,j=1,\dots,k+1$ and $j\neq i$}.
$$
We are now in position to apply Proposition \ref{prop:ChenLiIneq} obtaining the existence of a constant $C=C(\|\tilde H\|_\infty,\tilde\gamma,\bar r,\bar\gamma)$, such that
$$
\log\int_{\Sd}\tilde{H}e^{u_n} dV_{g_0}\leq\frac{1}{16(k+1)\pi-\tilde\gamma} \int_{\Sd}|\nabla u_n|^2 dV_{g_0}\,+\frac1{|\Sd|}\int_{\Sd} u_n\,dV_{g_0} +C,
$$
where $\tilde \gamma$ is chosen such that $a_n=\frac12-\frac{\lambda_n}{16(k+1)\pi-\tilde\gamma}\to a>0$. Thus finally
$$
E_{\lambda_n}(u_n)\geq a_n\int_{\Sd} |\nabla u_n|^2 dV_{g_0} -\lambda_n C,
$$
and so, since $\|u_n\|_{H^1(\Sd)}\to+\infty$, we have that $E_{\lambda_n}(u_n)\to +\infty$, but this is a contradiction against \eqref{star} and the fact that $I_{\lambda_n}(u_n)$ is bounded from above.

\

\textbf{Step 2:} $\mu=\sum_{i=1}^k t_i\delta_{q_i}$, for $k$ points ${q_i}\in S^+\setminus S_{\eps}$.

\

By our choice of $\eps$, $e^{u_n} \leq e^C$ in $S_{\eps}$, therefore the claim is proved.

\

\textbf{Step 3:} $\tilde{K}e^{u_n} \rightharpoonup \sum_{i=1}^k t_i\delta_{q_i}$.

\

We first notice that applying Theorem~\ref{thm:ChenLinostro}
\begin{eqnarray}\label{flower0}
\int_{\Sd} \tilde{H}e^{u_n}dV_{g_0} &=& \int_{\Sd} \tilde{K}e^{u_n} dV_{g_0}+ \int_{\Sd} (\tilde{H}-\tilde{K})e^{u_n} dV_{g_0}\nonumber\\
&=& 1 + \int_{S_{\eps}} (\tilde{H}-\tilde{K})e^{u_n} dV_{g_0} \leq 1 + |\Sd|(\eps + \|\tilde K\|_{L^\infty(\Sd)})e^C.
\end{eqnarray}
This fact, combined with Step 2, implies that
$$
\tilde{H} e^{u_n} \to 0 \quad \mbox{uniformly in $S_{\eps}$, as $n\to+\infty$},
$$
and in turn
\beq\label{flower}
e^{u_n} \to 0 \mbox{ and } \tilde{K}e^{u_n} \to 0\quad \mbox{ uniformly in $S_{\eps}$, as $n\to+\infty$}.
\eeq
Since,
$$
\tilde{K} e^{u_n}=\left\{\begin{array}{ll}
\left( \int_{\Sd} \tilde{H}e^{u_n}dV_{g_0} \right) \mu_n  \qquad & \text{in $S^+\setminus S_{\eps}$,}\\
\tilde{K}e^{u_n} \qquad  & \text{in $S_{\eps}$,}
\end{array}\right.
$$
by Step 2, \eqref{flower} and \eqref{flower0} we have that, up to a subsequence,
$$
\tilde{K} e^{u_n} \rightharpoonup \left(\lim_{n\to+ \infty} \int_{\Sd} \tilde{H}e^{u_n} dV_{g_0} \right) \sum_{i=1}^k t_i\delta_{q_i},
$$
with $q_i \in S^+\setminus S_{\eps}$ for $i=1,\dots,k$.

Thus finally, being $\int_{\Sd}\tilde{K}e^{u_n}dV_{g_0}=1$ and $\tilde H\geq\tilde K$, we get $\lim\limits_{n\to+ \infty} \int_{\Sd} \tilde{H}e^{u_n}dV_{g_0}=1$ completing the proof of Step 3.

\

\textbf{Step 4:} $\lambda_n \to \lambda_0 \in \Gamma(\diva_\ell) =\left\{8\pi r + \sum_{j=1}^{\ell}8\pi(1+\alpha_j)n_j\,|\,r\in \mathbb{N}\cup \{0\}, n_j\in\{0,1\}\right\}$.

\

By the previous steps we have:
$$
\lambda_n \tilde{K}(x) e^{u_n} \rightharpoonup \lambda_0 \sum_{i=1}^k t_i\delta_{q_i} \, \mbox{ in $\Sd$}.
$$

As in \cite{bt}, in order to characterize the possible values of $\lambda_0$ we use Green's representation formula on the solution $v_n=u_n-h_m$ of \eqref{eqdelta}, where $h_m$ is defined in \eqref{accame}. Observe that $\tilde{K}e^{u_n}=Ke^{v_n}$, so
\beq\label{concp}
\lambda_n Ke^{v_n} \rightharpoonup \lambda_0\sum_{i=1}^k t_i\delta_{q_i}, \, \mbox{ in $\Sd$}.
\eeq

Consequently, we derive that
$$
v_n-\frac{1}{|\Sd|}\int_{\Sd} v_n dV_{g_0}\to \sum_{i=1}^k t_i G(x,q_i)-h_m
$$
uniformly on a compact set of $\Sd\setminus\{ q_1,\ldots,q_k\}$, where $G(x,y)$ is the Green's function defined in \eqref{fgreen}.

Furthermore, the sequence $v_n-\frac{1}{|\Sd|}\int_{\Sd}v_n dV_{g_0}$ admits uniformly bounded mean oscillation on any compact subset of $\Sd\setminus \left(\{ q_1,\ldots, q_k\} \cup  \{p_1,\ldots,p_\ell\} \right)$. As a consequence, for every open subset $\Omega$ compactly contained in \linebreak$\Sd\setminus \left(\{ q_1,\ldots, q_k\} \cup  \{p_1,\ldots,p_\ell\} \right)$ there exists a uniform constant $C>0$ such that,
\beq\label{bmocond}
\displaystyle{\max_{\Omega} v_n - \min_{\Omega} v_n} \leq C.
\eeq

Finally, if $q_i\notin \{p_1,\ldots,p_\ell\}$, we can apply the local result of Y.Y.Li \cite{Li} to conclude that
\beq\label{quant1}
\lm_n \int_{B_{q_i}(r)}K(x) e^{v_n}  dV_{g_0} \to 8\pi,
\eeq
for every $r>0$ small enough.

Whereas in case that $q_i=p_j$ for some $i \in \{1,\ldots,k\}$ and $j \in \{1,\ldots,\ell\}$, since \eqref{bmocond} remains true, by Theorem~6 of \cite{bt} and \eqref{concp}, we get
\beq\label{quant2}
\lm_n \int_{B_{q_i}(r)}K(x) e^{v_n}  dV_{g_0} \to 8\pi(1+\alpha_j),
\eeq
for every $r>0$ small enough.

Thus, \eqref{quant1}, \eqref{quant2} and \eqref{concp} imply that
$$
\lm_0=\lim_{n\to+\infty}\lm_n= \lim_{n\to+\infty} \lm_n \int_{\Sd}K(x) e^{v_n} dV_{g_0} \in \Gamma(\diva_\ell).
$$

\

\textbf{Step 5:} Conclusion

\

By virtue of Step 4 we reach a contradiction with our assumptions. Therefore $u_n$ is bounded and as explained above, up to a subsequence, $u_n \to u_0 \in X$ which is a solution of \eqref{equation} for $\lambda=\lambda_0$.

\end{proof}

\section{Low sublevels of $I_\lambda$}\label{low}
\setcounter{equation}{0}
The aim of this Section is to define a map from low sublevels of $I_\lambda$ onto a non-contractible compact topological space and a reverse map from this space onto $I_\lambda$.\\
We will consider first the general case $\lambda\in(8\pi k,8\pi(k+1))$, $k\geq1$, and then we will focus on the case $k=1$ in which we will provide a more accurate characterization of the topology of the low sublevels which in the end will allow us to get existence of solutions also in some cases when $S^+$ is contractible.

\

We recall the notation: $I_{\lambda}^{a}=\{u\in X\,:\,I_\lambda(u)\leq a\}$.

\subsection{Construction of a continuous map \emph{from} low sublevels}
\subsubsection{\boldmath${\lambda\in(8\pi k,8\pi(k+1)), \; k\geq 1}$}
\

In the following we will use on ${(\Sd)}_k$, the set of formal barycenters of order $k$ on $\Sd$, see \eqref{formbar} for the definition, the metric given by $C^1(\Sd)^*$, which induces the same topology of the weak topology of distributions. Given $\sigma_1,\sigma_2 \in {(\Sd)}_k$ we will denote by $\dist(\sigma_1,\sigma_2)$ their distance and consistently with this convention we define the distance of an $L^1$ probability measure $f$ on $\Sd$ from a distribution $\sigma\in {(\Sd)}_k$ as
$$
\dist(f,\sigma)=\sup\left\{\left|\int_{\Sd} f\varphi\,dV_{g_0}-\langle\sigma,\varphi \rangle\right|\,:\,\|\varphi\|_{C^1(\Sd)}\leq 1\right\}
$$
where $\langle\sigma,\varphi\rangle$ stands for the duality product between $\mathcal D(\Sd)$ and the space of distributions. Recall that for $r>0$ and a subset $\Omega$ of $\Sd$ we set $(\Omega)^r=\{x\in \Sd \mid \dist(x,\Omega)<r\}$.

\begin{proposition}\label{prop:Psinoncontractible}
Let $\lambda\in(8\pi k,8\pi(k+1))$, $k\geq1$ and assume \ref{H1}, \ref{H2}, \ref{H3} to hold. Let $A_j$ be the $j$-th connected component of $S^+$, namely $S^+=\amalg_{j=1}^{N^+} A_j$, where the symbol $\amalg$ denotes the disjoint union.
\begin{itemize}
\item[\emph{(a)}] Let $P^{N^+}=\{\bar x_1,\ldots,\bar x_{N^+} \}\subset S^+\setminus\{p_1,\ldots,p_\ell\}$ with $\bar x_j\in A_j$ for any $j\in\{1,\ldots,N^+\}$, then for $L$ sufficiently large there exists a continuous projection
    $$
\Psi:I^{-L}_\lambda\to(P^{N^+})_k
$$
with the property that if $\frac{ e^{u_n}\chi_{S^+}}{\int_{S^+} e^{u_n}dV_{g_0}}\rightharpoonup\sigma$ for some $\sigma\in (P^{N^+})_k$, then $\Psi(u_n)\to\sigma$.
\item[\emph{(b)}] If $S^+$ has a connected component which is non-simply connected then for $L>0$ sufficiently large there exists a curve $\Gamma\subset S^+\setminus\{p_1,\ldots,p_\ell\}$ homeomorphic to $\mathbb{S}^1$ and a continuous projection
$$
\Psi:I^{-L}_\lambda\to\Gamma_k,
$$
with the property that if $\frac{ e^{u_n}\chi_{S^+}}{\int_{S^+} e^{u_n}dV_{g_0}}\rightharpoonup\sigma$ for some $\sigma\in\Gamma_k$, then $\Psi(u_n)\to\sigma$.
\end{itemize}
\end{proposition}

\begin{remark}\label{rem:PN+noncontr}
It is worth to point out that under assumption \ref{H4}, both $(P^{N^+})_k$ and $\Gamma_k$ are non contractible.

Indeed, the set $(P^{N^+})_k$ is the $(k-1)$-skeleton of a $(N^+-1)$-symplex and then it can be easily seen that it is not contractible if and only if $k<N^+$ (see for example Exercise 16 in Section 2.2 of \cite{Hatcher}). Whereas, to show that $\Gamma_k$ is non contractible, we can refer to \cite{BarDemMal} in which it is proved that the homology group $H_1(\Gamma;\mathbb{Z})\neq0$.

%
\end{remark}

\begin{proof}
\textbf{Step 1:} There exists $L>0$ sufficiently large and a continuous map
\[
\tilde\Psi: I_{\lambda}^{-L}\to {(\Sd)}_k,
\]
satisfying the following property:
\begin{itemize}
\item[\emph{(i)}] if $\frac{ e^{u_n}\chi_{S^+}}{\int_{S^+} e^{u_n}dV_{g_0}}\rightharpoonup\sigma$ for some $\sigma\in {(\Sd)}_k$, then $\tilde\Psi(u_n)\to\sigma$.
\end{itemize}
This follows directly from the proof of Lemma 4.9 of \cite{Djad}, just observing that $\tilde K\leq \max_{\Sd}(\tilde K)\chi_{S^+}$.

\textbf{Step 2:} For $L>0$ sufficiently large there exists a continuous map
\[
\tilde\Psi_{\delta}: I_{\lambda}^{-L}\to ((S^+)^\delta)_k
\]
with the property that if $\frac{ e^{u_n}\chi_{S^+}}{\int_{S^+} e^{u_n}dV_{g_0}}\rightharpoonup\sigma$ for some $\sigma\in (S^+)_k$, then $\tilde\Psi_{\delta}(u_n)\to\sigma$.

Let $X$ be a topological space and $Z$ a subspace of $X$, we recall that a continuous map $r:X\to Z$ is a deformation retraction of $X$ onto $Z$ if $r(z)=z$ for all $z\in Z$ and its composition with the inclusion is homotopic to the identity map on $X$. Clearly if such a map exists $X$ is homotopically equivalent to $Z$.

Let us fix $\delta>0$ sufficiently small such that $(A_i)^\delta\cap(A_j)^\delta=\emptyset$ for any $i\neq j$, where $A_1,\ldots,A_{N^+}$ are the connected components of $S^+$, and that $(S^+)^\delta$ is homotopically equivalent to $S^+$. This choice is possible by \ref{H2}. Moreover, making $\delta$ smaller, if necessary, we can suppose that $\partial(S^+\setminus (\partial S^+)^\delta)$ admits at each point a normal vector varying continuously and, furthermore, that $S^+\setminus (\partial S^+)^\delta$ is homotopically equivalent to $S^+$ and there exists
\[
r_0:(S^+)^\delta\to S^+\setminus (\partial S^+)^\delta
 \]
which is a deformation retraction of each $(A_i)^\delta$ onto $A_i\setminus(\partial A_i)^\delta$ and so of $(S^+)^\delta$ onto $S^+\setminus (\partial S^+)^\delta$. Besides by \ref{H3} we can suppose that $\{p_1,\ldots,p_\ell\}\subset S^+\setminus (\partial S^+)^\delta$.

Let us consider a cutoff function $\xi_{\delta}\in C^1(\Sd)$ such that
\[
\xi_\delta: \Sd \to[0,1],\quad \xi_{\delta |\overline{S^+}}\equiv0,\quad \xi_{\delta |\Sd\setminus(S^+)^\delta}\equiv1.
\]
Reasoning as in Step 1 of the proof of Theorem~\ref{thm:compactness} it is possible to show that if $L$ is sufficiently large
\begin{equation}\label{estimatedist}
\dist\left(\tilde\Psi(u),\frac{e^u\chi_{S^+}}{\int_{S^+}e^u\,dV_{g_0}}\right)\leq\frac1{4\|\xi_\delta\|_{C^1(\Sd)}}.
\end{equation}
Next, we want to show that for $L$ large enough and for any $u\in I^{-L}_\lambda$ if $\tilde\Psi(u)=\sum_{i}t_i\delta_{x_i}$ then $\sum_{i}t_i\xi_{\delta}(x_i)\leq\frac12$.\\
Indeed, if not, by \eqref{estimatedist}
\begin{eqnarray*}
\frac 1{4\|\xi_\delta\|_{C^1(\Sd)}}&\geq&\dist\left(\tilde\Psi(u),\frac{ e^{u}\chi_{S^+}}{\int_{S^+} e^{u}\,dV_{g_0}}\right)\\
&\geq&\left|\langle\sum_i t_i\delta_{x_i},\frac{\xi_\delta}{\|\xi_\delta\|_{C^1(\Sd)}}\rangle-\frac{\int_{\Sd} e^u \chi_{S^+} \xi_\delta dV_{g_0}}{\|\xi_\delta\|_{C^1(\Sd)}\int_{S^+} e^u dV_{g_0} }\right|\\
&=&\sum_i t_i\frac{\xi_\delta(x_i)}{\|\xi_\delta\|_{C^1(\Sd)}}\geq \frac 1{2\|\xi_\delta\|_{C^1(\Sd)}}
\end{eqnarray*}
which is a contradiction.

Then $\sum_i t_i(1-\xi_\delta(x_i))>\frac 12$ for any $u\in I^{-L}_\lambda$ and if we set
\[
\tilde\Psi_\delta(u)=\frac{\sum_i t_i (1-\xi_\delta(x_i))\delta_{x_i}}{\sum_i t_i (1-\xi_\delta(x_i))},
\]
then $\tilde\Psi_{\delta}$ is well defined and continuous.
The second property follows immediately from Step 1, the definition of $\xi_{\delta}$ and the definition of $\tilde\Psi_\delta$.

\textbf{Step 3:} Conclusion.

\emph{(a)}.

We define
\[
\begin{array}{ccccr}
\tilde\pi:&(S^+)^\delta&\longrightarrow &P^{N^+} &\\
 &x&\longmapsto&\bar x_j&\qquad\forall\,x\in (A_j)^\delta,
\end{array}
\]
and we notice that
\beq\label{Pigrecotilde=Id}
\tilde\pi_{|P^{N^+}}=Id_{|P^{N^+}}.
\eeq
We can finally define the map $\Psi$ as follows
\[
\begin{array}{cccc}
\Psi:&I^{-L}_\lambda&\longrightarrow&(P^{N^+})_k\\
 &u&\longmapsto&\sum_i s_i\delta_{\tilde\pi(y_i)}
\end{array}
\]
where $s_i$ and $y_i$ are defined by $\tilde\Psi_\delta$, namely $\tilde\Psi_{\delta}(u)=\sum_i s_i\delta_{y_i}$.\\
Clearly by Step 2, \eqref{Pigrecotilde=Id} and the definition of $\Psi$ if $\frac{ e^{u_n}\chi_{S^+}}{\int_{S^+} e^{u_n}dV_{g_0}}\rightharpoonup\sigma\in (P^{N^+})_k$ then $\Psi(u_n)\to\sigma$.\\
This proves point \emph{(a)}.

\

\emph{(b)}.

Let $q$ be a point in the interior of $\Sd \setminus (S^{+})^{\delta}$ and let $P:\Sd\setminus \{q\} \to \R^2$ the stereographic projection, defined in \eqref{setRP}. By our choice of $\delta$ and the assumption on $S^+$ there exists a regular closed curve $\gamma \subset \R^2$ such that at least one connected component $\Omega^-$ of $P(\Sd\setminus (S^+)^{\delta})$ lies in its interior.

By the Jordan-Sch\"{o}nflies Theorem (see \cite{Schonflies}) there exists an homeomorphism $\Phi:\R^2\to\R^2$ such that $\Phi(\gamma)=S^1$ and the interior of $\gamma$ is mapped onto the interior of $S^1$.
In turn fixing a point $x^- \in \Omega^-$ we can define a retraction $R:\R^2\setminus \{\Phi(x^-) \} \to S^1$. Then denoting by $\Gamma:=P^{-1}(\gamma)$, define
\[
\begin{array}{ccccr}
\pi:&(S^+)^\delta&\longrightarrow &\Gamma &\\
 &x&\longmapsto& P^{-1}\circ\Phi^{-1}\circ R \circ \Phi \circ P(x)
\end{array}
\]
which is well defined by the choice of q. Moreover since $R_{|{S^1}}=\Id_{S^1}$ we have that
\beq\label{Pigreco=Id}
\pi_{|{\Gamma}}=\Id_{|\Gamma}.
\eeq

We can finally define the map $\Psi$ as follows
\[
\begin{array}{cccc}
\Psi:&I^{-L}_\lambda&\longrightarrow&\Gamma_k\\
 &u&\longmapsto&\sum_i s_i\delta_{\pi(y_i)}
\end{array}
\]
where $s_i$ and $y_i$ are defined by $\tilde\Psi_\delta$, namely $\tilde\Psi_{\delta}(u)=\sum_i s_i\delta_{y_i}$.\\
Clearly by Step 2, \eqref{Pigreco=Id} and the definition of $\Psi$ if $\frac{ e^{u_n}\chi_{\Sigma^+}}{\int_{\Sigma^+} e^{u_n}dV_{g_0}}\rightharpoonup\sigma\in \Gamma_k$ then $\Psi(u_n)\to\sigma$.\\
\end{proof}

\begin{remark}\label{remnons}
It is possible to extend Proposition~\ref{prop:Psinoncontractible} to a general surface $\Sigma$. This is trivial in case $(a)$, while the generalization of point $(b)$ requires a more refined construction, involving a different compact subset $\Gamma\subset\{x\in\Sigma\,|\,K(x)>0\}\setminus\{p_1,\ldots,p_\ell\}$ homotopically equivalent to a connected, but not simply connected, component of $\{x\in\Sigma\,|\,K(x)>0\}$.
\end{remark}

\subsubsection{\boldmath${\lambda\in(8\pi,16\pi)}$}

\

The next results are helpful to treat the case when $\lambda \in (8\pi,16\pi)$ and $S^+$ is non contractible. Indeed in this situation \ref{H4} is not satisfied and so Proposition~\ref{prop:Psinoncontractible} does not provide a map from $I_{\lambda}^{-L}$ into a non contractible set, see Remark~\ref{rem:PN+noncontr}. Besides, we will construct, for $\lambda\in(8\pi,16\pi)$ a map which will allow to get a more accurate description of the low sublevels even for $S^+$ non contractible. In fact this more precise construction is not necessary to get existence of solutions for \eqref{equation} when $S^+$ has nontrivial homotopy type but we think that can be useful to obtain a better multiplicity result for solutions of \eqref{equation}. We plan to treat the multiplicity issue in a forthcoming paper. In conclusion, notice that in the following propositions is not assumed $S^+$ to be non contractible.

\begin{proposition}\label{prop:beta}
Assume $(\Sigma,g)=(\Sd,g_0)$, $p_1,\ldots,p_m\in\Sd$, $\alpha_1,\ldots,\alpha_m>0$ and $\lambda\in(8\pi,16\pi)$. If \ref{H1}, \ref{H2} hold and $C_1>2$ is a constant, then there exist $\tau>0$, $L_0>0$ and a continuous map
\beq\label{beta}
\beta:I_{\lambda}^{-L_0}\to \overline{S^+},
\eeq
satisfying the following property: for any $u\in I_{\lambda}^{-L_0}$ there exist $\bar \sigma>0$ and $\bar y\in \Sd$ such that $\dist(\bar y, \beta(u))<5 C_1 \bar \sigma$ and

\beq\label{ysigmaproperty}
\int_{B_{\bar y}(\bar \sigma)\cap S^+} \tilde K e^u\,dV_{g_0}=\int_{S^+\setminus B_{\bar y}(C_1 \bar\sigma)} \tilde K e^u\,dV_{g_0}\geq \tau\int_{S^+} \tilde K e^u\,dV_{g_0}.
\eeq
\end{proposition}

\begin{remark}\label{rem:below prop beta}
It is worth to point out that, even though Proposition 3.1 of \cite{MalchiodiRuizSphere} holds true also on a manifold with boundary, we can not apply directly such result because our functional $I_\lambda$ is defined on functions in $H^1(\Sd)$ and not in $H^1(S^+)$. However we will follow the arguments of \cite{MalchiodiRuizSphere} modifying them non trivially in order to handle the extra difficulty  given from the fact that $\tilde K$ changes sign and so $S^-$ has positive measure and $S^+$ is not necessarily connected.
\end{remark}

\begin{proof}
Let us define
$$
\mathcal A_0=\{f\in L^1(\Sd)\mid f(x)\geq0\textrm{ a.e., }\int_{\Sd}f\,dV_{g_0}=1\},
$$
$$
\sigma:\Sd\times\mathcal A_0\longrightarrow (0,+\infty),
$$
where $\sigma=\sigma(x,f)$ is such that
$$
\int_{B_x(\sigma)}f\,dV_{g_0}=\int_{\Sd\setminus B_x(C_1\sigma)} f\,dV_{g_0}.
$$
Notice that the value $\sigma(x,f)$ is \emph{not} uniquely determined, \emph{nor} necessarily continuous.

Now let us define $T:\Sd\times\mathcal A_0\longrightarrow(0,+\infty)$ by
$$
T(x,f)=\int_{B_x(\sigma(x,f))}f\,dV_{g_0}.
$$
Notice that $T(x,f)$ does not depend on $\sigma$ and it is uniquely determined.
\textbf{Step 0:} $T$ is continuous.

Let us suppose by contradiction that there exist $(x_n,f_n)\in \Sd\times \mathcal A_0$ such that
$$
(x_n,f_n)\to (x,f)\in \Sd\times \mathcal A_0\quad\textnormal{but}\quad  |T(x_n,f_n)-T(x,f)|\not\to 0\quad\textnormal{as $n\to+\infty$}.
$$
Being $0<\sigma(x_n,f_n)<\frac12 \diam(\Sd)$, up to a subsequence $\sigma(x_n,f_n)\to \sigma_\infty$, as $n\to+\infty$.\\
Now if $\sigma_\infty=\sigma(x,f)$, then
\beq\label{meas to 0}
\meas(B_{x_n}(\sigma(x_n,f_n))\vartriangle B_x(\sigma(x,f)))\to 0\qquad \textnormal{as $n\to+\infty$},
\eeq
and so by the convergence of $f_n$ to $f$ in $L^1$ we have
\begin{eqnarray}\label{Tcaso1}
|T(x,f)-T(x_n,f_n)|&\leq&\int\limits_{ B_{x}(\sigma(x,f))\setminus B_{x_n}(\sigma(x_n,f_n))}\!\!\!\!\!\!f\,dV_{g_0}+\int\limits_{B_{x_n}(\sigma(x_n,f_n))\setminus B_{x}(\sigma(x,f))} \!\!\!\!\!\!f_n\,dV_{g_0}\nonumber\\
&\;&+\int\limits_{B_{x_n}(\sigma(x_n,f_n))\cap B_x(\sigma(x,f))} \!\!\!\!\!\!|f_n-f|\,dV_{g_0}\stackrel{n\to+\infty}{\longrightarrow} 0,
\end{eqnarray}
which gives the desired contradiction.

On the other hand if $\sigma_\infty>\sigma(x,f)$, then for $n$ sufficiently large
\beq\label{inclusions}
\begin{array}{c}
B_x(\sigma(x,f))\subset B_{x_n}(\sigma(x_n,f_n))\\ \Sd\setminus B_x(C_1\sigma(x,f))\supset \Sd\setminus B_{x_n}(C_1\sigma(x_n,f_n)).
\end{array}
\eeq
Then for $n$ sufficiently large
\begin{equation*}
|T(x,f)-T(x_n,f_n)|\leq\int\limits_{B_{x_n}(\sigma(x_n,f_n))}|f_n-f|\,dV_{g_0}+\int\limits_{B_{x_n}(\sigma(x_n,f_n))\setminus B_{x}(\sigma(x,f))}f\,dV_{g_0}
\end{equation*}
and so in turn by the convergence of $f_n$ to $f$ we have that
\beq\label{appoepsilon}
\liminf_{n\to+\infty} \int\limits_{B_{x_n}(\sigma(x_n,f_n))\setminus B_{x}(\sigma(x,f))}f\,dV_{g_0} > 0.
\eeq
By the definition of $\sigma$, the convergence of $f_n$ to $f$ and \eqref{inclusions} we get
\begin{eqnarray*}
\int\limits_{B_x(\sigma(x,f))}f\,dV_{g_0}&=&\int\limits_{\Sd\setminus B_x(C_1(\sigma(x,f)))}f\,dV_{g_0}=\int\limits_{\Sd\setminus B_x(C_1(\sigma(x,f)))}f_n\,dV_{g_0}+o(1)\\
&\geq&\int\limits_{\Sd\setminus B_{x_n}(C_1(\sigma(x_n,f_n)))}f_n\,dV_{g_0}+o(1)=\int\limits_{B_{x_n}(\sigma(x_n,f_n))}f_n\,dV_{g_0}+o(1)\\
&=&\int\limits_{B_{x_n}(\sigma(x_n,f_n))}(f_n-f)\,dV_{g_0}+\int\limits_{B_{x_n}(\sigma(x_n,f_n))\setminus B_x(\sigma(x,f))}f\,dV_{g_0}\\
&&+\int\limits_{B_x(\sigma(x,f))}f\,dV_{g_0}+o(1)\\
&\geq&\int\limits_{B_{x_n}(\sigma(x_n,f_n))\setminus B_x(\sigma(x,f))}f\,dV_{g_0}+\int\limits_{B_x(\sigma(x,f))}f\,dV_{g_0}+o(1)
\end{eqnarray*}
which, combined with \eqref{appoepsilon}, gives the desired contradiction.

At last, the case $\sigma_\infty<\sigma(x,f)$ can be treated exactly as the latter case, just reversing the roles of $B_{x_n}(\sigma(x_n,f_n))$ and $B_x(\sigma(x,f))$.

\

\textbf{Step 1:} There exists $\tau>0$ such that $\max_{x\in \Sd}T(x,f)>2\tau$ for all $f\in\mathcal A_0$.

Let us introduce
$$
\mathcal A=\{h\in L^1(\Sd),\,h(x)>0\textrm{ a.e., }\int_{\Sd}h\,dV_{g_0}=1\}.
$$
It is easy to see that $\mathcal A$ is dense in $\mathcal A_0$. Moreover in Step 1 of Proposition 3.1 of \cite{MalchiodiRuizSphere} it is proved that there exists $\tilde\tau>0$ such that $\max_{x\in \Sd}T(x,f)>2\tilde\tau$ for all $f\in\mathcal A$. So our thesis follows from these facts and Step 0. Indeed, fix $f\in \mathcal{A}_0$ and let $\{h_n\}\subset \mathcal A$ such that $h_n\to f$ in $L^1(\Sd)$ and let $x_n\in \Sd$ such that $T(x_n,h_n)=\max_{x\in \Sd}T(x,h_n)$, then $T(x_n,h_n)>2\tilde\tau$. Up to a subsequence $x_n\to x_0\in \Sd$ as $n\to+\infty$ and so, by the continuity of $T$, $T(x_n,h_n)\to T(x_0,f)\geq 2\tilde \tau$. The thesis follows taking $\tau=\frac{\tilde\tau}2$.

\

\textbf{Step 2:} Let us define
$$
S(f)=\{x\in \Sd\mid T(x,f)\geq\tau\}.
$$
By Step 0 and Step 1 $S(f)$ is a non empty compact set for any $f\in\mathcal{A}_0$.

Let us define also
$$
\bar{\sigma}(f)=\sup_{x\in S(f)} \sigma(x,f).
$$
Let us prove that even if $\sigma$ is not continuous, up to eventually redefine $\sigma(\cdot,f)$ in a point, there exists
$$\bar y\in S(f)\quad \textnormal{such that}\quad\sigma(\bar y,f)=\bar\sigma.$$
Indeed let $\{x_n\} \subset S(f)$ such that $\sigma(x_n,f)\to\bar\sigma(f)$, then since $S(f)$ is compact, up to a subsequence, $x_n\to\bar y\in S(f)$. Thus
$$
\int_{B_{x_n}(\sigma(x_n,f))}f\,dV_{g_0}=\int_{\Sd\setminus B_{x_n}(C_1\sigma(x_n,f))} f\,dV_{g_0}\quad
$$
and so
$$
\int_{B_{\bar y}(\bar\sigma(f))}f\,dV_{g_0}=\int_{\Sd\setminus B_{\bar y}(C_1\bar\sigma(f))} f\,dV_{g_0}.
$$
Now if $\sigma(\bar y,f)<\bar\sigma(f)$ we can redefine $\sigma(\cdot,f)$ at $\bar y$ as $\sigma(\bar y,f)=\bar\sigma(f)$, and the proof of our claim is completed. Clearly this modification does not affect the previous steps.

\

For $u\in X$, take $f\equiv f_u=\frac{\tilde K e^u \chi_{{S^+}}}{\int_{{S^+}}\tilde K e^u\,dV_{g_0}}$.

\

\textbf{Step 3:} For any $\eps>0$ there exists $L_0>0$ large enough such that $\diam S(f)\leq (C_1+1)\bar \sigma<\eps$ for all $u\in I^{-L_0}_\lambda$.

By definition of $\bar\sigma$, $S(f)$ and $S^+$
$$
\int_{B_{\bar y}(\bar\sigma)\cap {S^+}}\tilde K e^u\,dV_{g_0}\geq\tau\int_{{S^+}}\tilde K e^u\,dV_{g_0}\geq \tau\int_{\Sd}\tilde K e^u\,dV_{g_0}\textrm{\quad and}$$
$$
\int_{{S^+}\setminus B_{\bar y}(C_1\bar\sigma)}\tilde K e^u\,dV_{g_0}\geq\tau\int_{{S^+}}\tilde K e^u\,dV_{g_0}\geq \tau\int_{\Sd}\tilde K e^u\,dV_{g_0}.
$$
Then Proposition \ref{prop:ChenLiIneq} implies that $\bar\sigma\to0$, as $L\to+\infty$, uniformly for $u\in I^{-L}_\lambda$. Thus we can choose $L_0>0$ such that $\bar\sigma<\min\left\{\frac{\eps}{C_1+1},\frac{\min_i(\diam D_i)}{6}\right\}$ for any $u\in I_{\lambda}^{-L}$, where $D_i$ are the connected components of $S^+$.

Now take $x$, $y\in S(f)$, where $f=\frac{\tilde K e^u \chi_{{S^+}}}{\int_{{S^+}}\tilde K e^u\,dV_{g_0}}$, $u\in I_{\lambda}^{-L_0}$, we claim that
\beq\label{claimdist}
\dist(x,y)\leq C_1\max\{\sigma(x,f),\sigma(y,f)\}+\min\{\sigma(x,f),\sigma(y,f)\}.
\eeq
Let us prove \eqref{claimdist}.\\
Let us suppose by contradiction that $B_x(C_1(\sigma(x,f)))\cap B_y(\sigma(y,f)+\eps)=\emptyset$ for some $\eps>0$. Clearly we can take $\eps<\frac{\min_i(\diam D_i)}{6}$ and such that $B_y(\sigma(y,f)+\eps)$ does not exhaust the whole $S^+$. Let us now show that $A_y(\sigma(y,f),\sigma(y,f)+\eps)\cap S^+$ is a nonempty open set. \\
Let us prove first that there exists $z\in \partial B_y(\sigma(y,f)+\eps)\cap S^+$.\\
By contradiction we suppose that $\partial B_y(\sigma(y,f)+\eps)\cap S^+=\emptyset$.\\
Since $\int_{B_y(\sigma(y,f))\cap S^+}f\,dV_{g_0}>0$, $B_y(\sigma(y,f)+\eps)\cap S^+\neq\emptyset$, so $D_i\subset B_y(\sigma(y,f)+\eps)$ for some $i$. This would imply that $\min_i (\diam (D_i))<2(\sigma(y,f)+\eps)\leq 2\bar\sigma+2\eps<\frac23 \min_i (\diam (D_i))$ which is impossible.

Next, being $S^+$ open, $B_z(\eps)\cap A_y(\sigma(y,f),\sigma(y,f)+\eps)\cap S^+$ is a nonempty open set.
Then
\begin{eqnarray*}
\int_{B_x(\sigma(x,f))\cap S^+}\tilde K e^u \,dV_{g_0}&=&\int_{S^+\setminus B_x(C_1\sigma(x,f))}\tilde K e^u\,dV_{g_0}\\
&\geq&\int_{B_y(\sigma(y,f)+\epsilon)\cap S^+}\tilde K e^u\,dV_{g_0}>\int_{B_y(\sigma(y,f))\cap S^+}\tilde K e^u\,dV_{g_0}.
\end{eqnarray*}

By interchanging the roles of $x$ and $y$, we would also obtain the reverse inequality. This contradiction proves \eqref{claimdist}.

Then by \eqref{claimdist} and the definition of $\bar\sigma$ we have $\dist(x,y)\leq (C_1+1)\bar\sigma$ for any given $x,y\in S(f)$.

\

\textbf{Step 4:} Definition of $\beta$ and conclusion.

We consider $\Sd$ embedded in $\R^3$ and we define
$$
\eta: I_\lambda^{-L_0}\to\R^3, \; \eta(u)=\frac{\int_{\Sd} [T(x,f)-\tau]^+ x\,dV_{{g_0}}}{\int_{\Sd} [T(x,f)-\tau]^+ dV_{{g_0}}}\quad \textnormal{where $f\equiv f_u=\frac{\tilde K e^u \chi_{{S^+}}}{\int_{{S^+}}\tilde K e^u\,dV_{g_0}}$.}
$$
Notice that in the above terms the integrands vanish outside $S(f)$.\\

From now on, for $r>0$, according to our notation we will denote by $(S^+)^{r}=\{x\in \Sd \,|\,\dist(x,S^+)<r\}$.\\
Clearly $B_{\bar y}(\bar\sigma)\cap S^+\neq\emptyset$, namely
\beq\label{z}
\bar y\in(S^+)^{\bar\sigma},
\eeq
moreover by Step 3 $\diam(S(f))\leq (C_1+1)\bar\sigma$ and therefore being $\bar y\in S(f)$
$$
S(f)\subset (S^+)^{(C_1+2)\bar\sigma}\qquad\textnormal{and}\qquad S(f)\subset \bar B_{\bar y}^{\R^3}((C_1+1)\bar\sigma).
$$
Being $\eta(u)$ a barycenter of a function supported in $S(f)$, we have
\begin{equation}\label{u}
|\eta(u)-\bar y|\leq (C_1+1)\bar\sigma.
\end{equation}
Let $U\supset \Sd$, $U\subset\R^3$ an open tubular neighborhood of $\Sd$, and $P:U\to \Sd$ an orthogonal projection onto $\Sd$. Moreover by Step 3 there exists $L_0>0$ sufficiently large such that $\eta(u)\in U$ for any $u\in I_\lambda^{-L_0}$. Thus we can define
$$
\tilde \beta: I_\lambda^{-L_0}\to \Sd \qquad\qquad \tilde \beta(u)=P\circ \eta(u).
$$
Next, we claim that, eventually for a larger $L_0$,
\beq\label{zz}
\dist(\bar y,\tilde \beta(u))\leq 2C_1\bar\sigma.
\eeq

Let $T_{\bar{y}}(\Sd)$ be the tangent space to $\Sd$ at $\bar y$. For any $x \in S(f) \subset B^{\mathbb{R}^3}_{\bar y} ((C_1+1) \bar \sigma)$, we have that
$$
\min\left\{|\bar y + y -x | : y \in T_{\bar{y}}(\Sd) \right\} \leq C \bar{\sigma}^2,
$$
where $C$ depends only on the $C^2$ regularity of $\Sd$. Since $\eta(u)$ is a barycenter of a function supported in $S(f)$, it is clear that
$$
\min\left\{|\bar y + y -\eta(u) | : y \in T_{\bar{y}}(\Sd) \right\} \leq C \bar{\sigma}^2.
$$

By taking a larger $L_0$, if necessary, by Step~3 $\bar{\sigma}$ is small enough such that
\begin{equation}\label{uu}
|\tilde{\beta}(u)-\eta(u)| = \min_{x\in \Sd} |\eta(u)-x|\leq 2C \bar{\sigma}^2 \leq \bar{\sigma}.
\end{equation}

Since $C_1>2$, let $\nu = \frac{2C_1}{C_1+2}>1$, again, by Step~3 we can take $L_0$ large enough such that $\bar{\sigma}$ satisfies that for $x,y \in \Sd$, if $|x-y|\leq(C_1+2)\bar\sigma$, then $\dist(x,y)\leq \nu |x-y|$. This together with \eqref{u} and \eqref{uu} proves \eqref{zz}.

Combining \eqref{z} and \eqref{zz} we obtain that
\beq\label{zzz}
\dist(\tilde\beta(u),\overline{S^+})<(2C_1+1)\bar\sigma.
\eeq
Besides by the regularity of $\partial S^+$ there exists $\gamma>0$ and a continuous projection $\pi:(S^+)^{\gamma}\to\overline{S^+}$ such that
\beq\label{zzzz}
\pi_{|\overline{S^+}}=Id_{|\overline{S^+}}\qquad\textnormal{and}\qquad \dist(x,\pi(x))=\dist(x,\overline{S^+}).
\eeq
Again for $L_0>0$ large enough $2(C_1+1)\bar\sigma<\gamma$ and so, by \eqref{zzz}, $\tilde\beta(I_\lambda^{-L_0})\subset(S^+)^\gamma$. Then we can define $\beta:I_\lambda^{-L_0}\to \overline{S^+}$ as
$$
\beta(u)=\pi\circ \tilde\beta(u).
$$
At last by \eqref{zz}, \eqref{zzzz}, \eqref{zzz} and $C_1>2$ we have
\begin{eqnarray*}
\dist(\bar y,\beta(u))&\leq&\dist(\bar y, \tilde\beta(u))+\dist(\tilde\beta(u),\pi\circ\tilde\beta(u))\\
&\leq& 2C_1\bar\sigma+\dist(\tilde\beta(u),\overline{S^+})\\
&\leq&(4C_1+1)\bar\sigma< 5C_1\bar\sigma.
\end{eqnarray*}
\end{proof}

\begin{remark}\label{rem:betauntox}
With the above construction, if $f_n=\frac{\tilde K e^{u_n}\chi_{S^+}}{\int_{S^+}\tilde K e^{u_n}dV_{g_0}}\rightharpoonup \delta_x$ for some $x\in\overline{S^+}$ then one also has $\beta(u_n)\to x$. This can be seen exactly as in Remark~3.2 of \cite{MalchiodiRuizSphere}.
\end{remark}

Next we show that the functional $I_\lambda$ is bounded from below on the functions in $\beta^{-1}(J_\lambda)$, where $\beta$ is the map constructed in Proposition~\ref{prop:beta} and $J_\lambda$ is defined in \eqref{Jlambda}.

\begin{proposition}\label{prop:betaJlambda}
Assume $(\Sigma,g)=(\Sd,g_0)$, $p_1,\ldots,p_m\in\Sd$, $\alpha_1,\ldots,\alpha_{\ell}\in(0,1]$, $\alpha_{\ell+1},\ldots,\alpha_m>0$ and $\lambda \in (8\pi,16\pi)$. If \ref{H1}, \ref{H2} and \ref{H3} hold, then there exist $C_1>0$ sufficiently large, $L_0>0$, $\tau>0$ such that Proposition \ref{prop:beta} applies and there exists $L>L_0$ such that $I_\lambda(u)>-L$ for any $u\in I_\lambda^{-L_0}$ satisfying that $\beta(u)=p_i\in J_{\lambda}$.
\end{proposition}
\begin{proof}
We will follow very closely the proof of Proposition~4.1 in  \cite{MalchiodiRuizSphere}, adapting it to our different definition of $\beta$.

Let $\eps>0$ to be fixed later depending only on $\lambda$ and a universal constant $C_0$. In turn let $C_1>4$ large enough so that $\eps^{-1}+1<\log_4 C_1$ and let $L_0>0$ and $\tau>0$ such that Proposition~\ref{prop:beta} applies.

Let us suppose by contradiction that there exists a sequence $u_n\in X$ such that $I_\lambda(u_n)\to-\infty$ and $\beta(u_n)=p_i\in J_\lambda$ as $n\to+\infty$. Clearly we can assume without loss of generality that $\int_{\Sd} u_n dV_{g_0}=0$.\\
Let ${\bar y}_n\in S$, ${\bar \sigma}_n>0$ be as in Proposition \ref{prop:beta}, such that $\dist({\bar y}_n, p_i)<5 C_1 {\bar \sigma}_n$. It is easy to see, applying Proposition \ref{prop:ChenLiIneq} as in Step 1 of the proof of Theorem~\ref{thm:compactness}, that ${\bar\sigma}_n\to0$.
Consequently, by virtue of \ref{H3}, for $n$ large enough $\bar{y}_n \in \Sd$. Then we fix $\delta>0$, smaller than the injectivity radius and such that $B_{{\bar y}_n}(\delta)\subset S^+$ for any $n$ sufficiently large, and we choose
\beq\label{epsilongammaC1}
N \in \N \qquad \textnormal{such that}\qquad\eps^{-1}<N<\log_4{C_1}.
\eeq
Since ${\bar\sigma}_n\to 0$ we have that for $n$ sufficiently large $C_1{\bar\sigma}_n<\delta$ and so
$$
\cup_{m=1}^N A_{{\bar y}_n}(4^{m-1}{\bar\sigma}_n,4^{m}{\bar\sigma}_n)\subset A_{{\bar y}_n}(\bar\sigma_n, C_1\bar\sigma_n)\subset B_{{\bar y}_n}(\delta).
$$
Then there exists $s_n\in[2{\bar \sigma}_n,\frac{C_1}2{\bar \sigma}_n ]$ such that

\beq\label{annulusgammaestimate}
\int_{A_{{\bar y}_n}(\frac{s_n}2,2s_n)}|\nabla u_n|^2\,dV_{g_0}\leq\frac 1{N}\int_{B_{{\bar y}_n}(\delta)}|\nabla u_n|^2\,dV_{g_0}.
\eeq

From now on, in order to simplify the notation, we drop the dependence on $n$.

Let us define
\[
\mathcal D_1=\int_{B_{\bar y}(s)}|\nabla u|^2 dV_{g_0},\qquad  \mathcal D_2=\int_{\Sd\setminus B_{\bar y}(s)}|\nabla u|^2 dV_{g_0}, \qquad  \mathcal D=\mathcal D_1+\mathcal D_2.
\]
The proof proceeds in three steps.

\

\textbf{Step 1:} We apply Proposition \ref{prop:MTlocalized} to a convenient dilation of $u$ given by
$$
v(x)=u(sx+\bar y).
$$
We have
$$
\int_{B_{\bar y}(s)}|\nabla u|^2 dV_{g_0}=\int_{B_0(1)}|\nabla v|^2 dV_{g_0}, \qquad \dashint_{B_{\bar y}(s)}u dV_{g_0}=\dashint_{B_0(1)}v dV_{g_0},
$$
\begin{eqnarray*}
\int_{B_{\bar y}(\frac s2)\cap S^+} \tilde K e^u dV_{g_0}&\leq& C\int_{B_{\bar y}(\frac s2)\cap S^+}|x-p_i|^{2\alpha_i}e^u\,dV_{g_0}\\
&\leq& C s^{2\alpha_i}\int_{B_{\bar y}(\frac s2)\cap S^+}e^u dV_{g_0}\leq C s^{2\alpha_i+2}\int_{B_{0}(\frac12)}e^v dV_{g_0}.
\end{eqnarray*}
In the above computations we have used that $|\bar y-p_i|\leq Cs$.
Then, recalling that by definition of $\tau$ (see Proposition \ref{prop:beta})
$$\int_{B_{\bar y}(\frac s2)\cap S^+}\tilde K e^u dV_{g_0} \geq\tau\int_{S^+} \tilde K e^u dV_{g_0}\geq \tau\int_{\Sd} \tilde K e^u dV_{g_0}$$
and applying Proposition \ref{prop:MTlocalized} to $v$ (with $\tilde H=1$) we get
\begin{eqnarray}\label{firstestimate}
\log\int_{\Sd} \tilde K e^u\,dV_{g_0}&\leq& C+2(1+\alpha_i)\log s+\log\int_{B_0(\frac 12)}e^v\,dV_{g_0}\nonumber\\
&\leq& C+2(1+\alpha_i)\log s+\frac1{16\pi-\eps}\int_{B_0(1)}|\nabla v|^2\,dV_{g_0}+\dashint_{B_0(1)} v\,dV_{g_0}\\
&=& C+2(1+\alpha_i)\log s+\frac1{16\pi-\eps}\mathcal D_1+\dashint_{B_{\bar y}(s)} u\,dV_{g_0}.\nonumber
\end{eqnarray}

\

\textbf{Step 2:} Exactly as in Proposition 4.1 of \cite{MalchiodiRuizSphere}, we estimate $\dashint_{\partial B_{\bar y}(s)}u\,dV_{g_0}$. By the trace embedding $\tilde u=u-\dashint_{B_{\bar y}(s)}u\,dV_{g_0}\in L^1(\partial B_{\bar y}(s))$ and thanks to the Poincaré-Wirtinger inequality we get
$$
\left|\dashint_{\partial B_{\bar y}(s)}\tilde u\,dx\right|\leq C\|\tilde u\|_{H^1}\leq C \left(\int_{B_{\bar y}(s)}|\nabla u|^2\,dV_{g_0}\right)^{\frac12}.
$$
Therefore,
\beq\label{appo3}
\left|\dashint_{\partial B_{\bar y}(s)}u\,dx-\dashint_{B_{\bar y}(s)}u\,dV_{g_0}\right|\leq C\left(\int_{B_{\bar y}(s)}|\nabla u|^2\,dV_{g_0}\right)^{\frac12}\leq \eps \mathcal D_1+C'.
\eeq
Now notice that, since the above inequality is invariant under dilation, the constant $C$ is independent of $s$ and hence $C'$ depends only on $\eps$.

\

\textbf{Step 3:} By virtue of the fact that $\tilde K(x)\sim \distd(x,p_i)^{2\alpha_i}$ near $p_i$, and $|x-p_i|\leq C|x-\bar y|$ in $S^+\setminus B_{\bar{y}}(s)$, we get the following estimate
\begin{eqnarray}\label{appo}
&\displaystyle{\int_{{S^+}\setminus B_{\bar y}(s)}\tilde K e^u\,dV_{g_0}}=&\int_{{S^+}\setminus B_{\bar y}(s)}\frac{\tilde K(x)}{|x-\bar y|^{2\alpha_i}} |x-\bar y|^{2\alpha_i} e^u\,dV_{g_0}\leq \\
&\displaystyle{\frac C{s^{2\alpha_i}}\int_{{S^+}\setminus B_{\bar y}(s)}e^{\hat v}\,dV_{g_0}\nonumber }\leq& \frac C{s^{2\alpha_i}}\int_{\Sd}e^{\hat v}\,dV_{g_0},
\end{eqnarray}
where $\hat v(x)=\hat u(x)+4\alpha_i w(x)$,
$$
w(x)=\left\{
       \begin{array}{ll}
        \log s & \hbox{$x\in B_{\bar y}(s)$,} \\
        \log |x-\bar y| & \hbox{$x\in A_{\bar y}(s,\delta)$,} \\
        \log \delta & \hbox{$\Sd\setminus B_{\bar y}(\delta)$,}
       \end{array}
     \right.
\qquad\quad
\left\{
  \begin{array}{ll}
    -\Delta_{g_0} \hat u=0 & \hbox{$x\in B_{\bar y}(s)$,} \\
    \hat u(x)=u(x) & \hbox{$x\notin B_{\bar y}(s)$}.
  \end{array}
\right.
$$
In order to apply the Moser-Trudinger inequality to $\hat v$ we observe that
\beq\label{dashinthatv}
\dashint_{\Sd} \hat v\,dV_{g_0}\leq C+\dashint_{\Sd} \hat u.
\eeq
Since $\dashint_{\Sd} u\,dV_{g_0}=0$ and $\hat u-u$ is compactly supported in $B_{\bar y}(s)$,
\beq\label{dashinthatu}
\left|\dashint_{\Sd} \hat u\,dV_{g_0}\right|=\left|\dashint_{\Sd} (\hat u-u)\right|\leq C\left(\int_{B_{\bar{y}}(s)}|\nabla \hat u-\nabla u|^2\,dV_{g_0}\right)^{\frac12}\leq \eps \mathcal D+C_\eps.
\eeq
We now estimate using \eqref{annulusgammaestimate} and \eqref{epsilongammaC1} the Dirichlet energy
\beq\label{dirichletthatv1}
\int_{B_{\bar y}(s)}|\nabla \hat v|^2\,dV_{g_0}=\int_{B_{\bar y}(s)}|\nabla \hat u|^2\,dV_{g_0}\leq C_0\int_{A_{\bar y}(\frac{s}2,2s)}|\nabla u|^2\,dV_{g_0}\leq C_0\eps \mathcal D,
\eeq
where $C_0$ is independent on the radius $s$, since everything is dilation invariant.

On the other hand integrating by parts we obtain
\begin{eqnarray}\label{dirichletthatv2}
\int_{\Sd\setminus B_{\bar y}(s)}|\nabla \hat v|^2\,dV_{g_0}&=&\int_{\Sd\setminus B_{\bar y}(s)}|\nabla \hat u|^2\,dV_{g_0}+16\alpha_i^2\int_{\Sd\setminus B_{\bar y}(s)}\frac1{|x-\bar y|^2}dV_{g_0}\nonumber\\
&\;&+8\alpha_i\int_{\Sd\setminus B_{\bar y}(s)}\nabla u\cdot\nabla(\log|x-\bar y|)\,dV_{g_0}\\
&\leq&\mathcal D_2-32\pi\alpha^2_i\log s-16\pi\alpha_i\dashint_{\partial B_{\bar y}(s)}u\,dV_{g_0}+C.\nonumber
\end{eqnarray}

Finally applying to $\hat v$ the Moser-Trudinger inequality, Proposition \ref{prop:MT}, and in turn \eqref{dirichletthatv1}, \eqref{dirichletthatv2}, \eqref{dashinthatv}, \eqref{dashinthatu} we get
\begin{eqnarray}\label{appo2}
\log\int_{\Sd} e^{\hat v}\,dV_{g_0}&\leq&\frac1{16\pi}\int_{B_{\bar y}(s)}|\nabla \hat v|^2\,dV_{g_0}+\frac1{16\pi}\int_{\Sd\setminus B_{\bar y}(s)}|\nabla \hat v|^2\,dV_{g_0}+\dashint_{\Sd} \hat v\,dV_{g_0}+C\nonumber\\
&\leq& \frac{C_0\eps \mathcal D}{16\pi}+\frac{\mathcal D_2}{16\pi}-2\alpha^2_i\log s-\alpha_i\dashint_{\partial B_{\bar y}(s)}u\,dV_{g_0}+\eps \mathcal D+C.
\end{eqnarray}
Now, recalling that $B_{\bar y}(s)\subset B_{\bar y}(C_1\bar \sigma)$, the definition of $\bar y$ (see Proposition \ref{prop:beta}), \eqref{appo} and \eqref{appo2} we have that
\begin{eqnarray}\label{secondestimate}
\log\int_{\Sd} \tilde K e^u\,dV_{g_0}&\leq& \log\int_{S^+}\tilde K e^u\,dV_{g_0}\leq \log\left(\frac1\tau\int_{S^+\setminus{B_{\bar y}(s)}}\tilde K e^u\,dV_{g_0}\right)\nonumber\\
&\leq&-2\alpha_i(1+\alpha_i)\log s+C_0\eps \mathcal D+\frac{\mathcal D_2}{16\pi}-\alpha_i \dashint_{\partial B_{\bar y}(s)}u\,dV_{g_0}+C.
\end{eqnarray}
At last, adding \eqref{firstestimate} (multiplied by $\alpha_i$) to \eqref{secondestimate} and using \eqref{appo3} and the assumption $\alpha_i\leq1$ we have
$$
(\alpha_i+1)\log\int_{\Sd} \tilde K e^u\,dV_{g_0}\leq \left(\frac{1}{16\pi-\eps}+C_0\eps\right)\mathcal D+C,
$$
so plugging this estimate in the functional we derive that
$$
I_\lambda(u)\geq \left(\frac12-\lambda\left(\frac1{(16\pi-\eps)(\alpha_i+1)}+\frac{C_0}{\alpha_i+1}\eps\right)\right)\int_{\Sd}|\nabla u_n|\,dV_{g_0}-C.
$$
In order to conclude it suffices to take $\eps$ small enough, depending only on $\lambda$ and $C_0$ ($C_0$ is a universal constant), such that $\left(\frac12-\lambda\left(\frac1{(16\pi-\eps)(\alpha_i+1)}+\frac{C_0}{\alpha_i+1}\eps\right)\right)>0$. Indeed, recalling that we were working with a sequence $u_n$, we get $I_\lambda(u_n)\geq-C$ which leads to the desired contradiction.
\end{proof}

Let $J_\lambda$ be as in \eqref{Jlambda} and let us fix a small positive number $\theta$ such that $S^+\setminus(\partial S^+)^\theta$ is a strong deformation retract of $\overline{S^+}$ and such that, for any $p_i\in J_\lambda$, $B_{p_i}(\theta)\subset S^+\setminus (\partial S^+)^\theta$. Such a $\theta$ does exist if we assume \ref{H2} and \ref{H3} to hold. Then we set

\beq\label{Thetalambda}
\Theta_{\lambda}=S^+\setminus\left((\partial S^+)^\theta\cup\bigcup\limits_{p_i\in J_\lambda}B_{p_i}(\theta)\right)
\eeq
and we finally conclude, defining a continuous map from $I^{-L}_\lambda$ to $\Theta_\lambda$.

\begin{proposition}\label{prop:Psicontractible}
Assume $(\Sigma,g)=(\Sd,g_0)$, $p_1,\ldots,p_m\in\Sd$, $\alpha_1,\ldots,\alpha_m>0$ and $\lambda \in (8\pi,16\pi)$. If \ref{H1}, \ref{H2} and \ref{H3} hold, then for $L>0$ sufficiently large there exists a continuous projection
$$
\Psi:I^{-L}_\lambda\to\Theta_\lambda
$$
with the property that if $\frac{\tilde K e^{u_n}\chi_{S^+}}{\int_{S^+}\tilde K e^{u_n}\,dV_{g_0}}\rightharpoonup\delta_x$ for some $x\in\Theta_\lambda$ then $\Psi(u_n)\to x$.
\end{proposition}

\begin{remark}\label{rem:Thetalambda}
Let us observe that $\Theta_{\lambda_0}$ is non contractible if and only if either $S^+$ is non contractible or $N^+>1$, namely \ref{H4} holds, or if $J_{\lambda_0} \neq \emptyset$, namely \ref{H5} holds. Moreover notice that $\Theta_{\lambda}=\Theta_{\lambda_0}$ for any $\lambda$ sufficiently close to $\lambda_0$.
\end{remark}

\begin{proof}
The proof mimics, with minor changes that of Proposition 4.4 of \cite{MalchiodiRuizSphere}, we just sketch it for reader's convenience.\\
Let us consider the map $\beta$ constructed in Proposition \ref{prop:beta}, then by Proposition \ref{prop:betaJlambda} if $\beta(u)\in J_\lambda$, $I_\lambda$ is uniformly bounded from below, therefore if $L$ is sufficiently large and if $u\in I^{-L}_\lambda$, then $\beta(u)\in \overline{S^+}\setminus J_\lambda$.

If $\beta(u)\in \Theta_\lambda$ we set $\Psi(u)=\beta(u)$, whereas if $\beta(u)\not\in \Theta_\lambda$, either it belongs to a subset of the form $B_{p_i}(\theta)\setminus\{p_i\}$ or it belongs to $\overline{S^+}\cap (S^+)^{\theta}$. In the first case we move $\beta(u)$ along the geodesic segment emanating from $p_i$ in the direction of $\beta(u)$ until we hit the boundary of $\Theta_\lambda$ and we set $\Psi(u)$ to be this point. In the second case instead we move $\beta(u)$ following the deformation retraction of $\overline{S^+}$ onto $S^+\setminus (\partial S^+)^\theta$. This procedure is well defined if $\theta$ is chosen sufficiently small and in particular such that $B_{p_i}(\theta)\subset S^+\setminus (\partial S^+)^\theta$ for any $p_i\in J_\lambda$.

The last statement follows from Remark \ref{rem:betauntox}.
\end{proof}

\begin{remark}\label{remnons2}
The arguments of the proof of Proposition~\ref{prop:beta}, Proposition~\ref{prop:betaJlambda} and Proposition \ref{prop:Psicontractible} work perfectly well for any compact surface $\Sigma$. The only modification needed is to consider, in Step 4 of Proposition \ref{prop:betaJlambda}, an isometrical embedding of $\Sigma$ in $\R^k$ rather than in $\R^3$.
\end{remark}
\

\subsection{Construction of continuous map \emph{into} low sublevels}
\

We will distinguish the case $\lambda\in(8\pi k,8\pi(k+1))$, $k\geq2$, from $\lambda\in(8\pi,16\pi)$.

\subsubsection{\boldmath${\lambda\in(8\pi k,8\pi (k+1)), k\geq 2}$}
\

Let $\lambda \in (8\pi k,8\pi (k+1))$, with $k\geq2$, and let $Y$ be a compact subset of $S^+\setminus\{p_1,\ldots,p_\ell\}$.

At first, we construct functions with arbitrary low energy.
For $b > 0$ to be fixed, small enough, we consider the smooth non-decreasing cut-off function $\chi_{b}: \mathbb{R}^+ \to \mathbb{R}^+$ such that

\beq\label{chib}
 \chi_{b}(t) =
  \begin{cases}
   t & \text{for } t \in \left[0,b\right], \\
   2b       & \text{for } t\geq 2b.
  \end{cases}
\eeq

For $\mu>0$ and $\sigma=\sum_{i=1}^k t_i \delta_{x_i} \in Y_k$, where $Y_k$ is the set of formal barycenters of order $k$ defined on $Y$, see \eqref{formbar}, we define

\beq\label{phimu}
 \varphi_{\mu,\sigma}: \Sd \to \R, \quad \ \varphi_{\mu,\sigma}(x)=\log\sum_{i=1}^k t_i\left(\frac{\mu}{1+(\mu \chi_{b}(dist(x,x_i)))^2}\right)^2.
\eeq
\
It can be easily seen that for $b$ sufficiently small and for $\mu$ large enough \linebreak $\left\{\varphi_{\mu,\sigma} \mid \sigma \in Y_k \right\} \subset X$, where $X$ is introduced in \eqref{X}, and, noticing that $\tilde K$ is strictly positive on $Y$, we can argue
as in \cite{Djad} to obtain the following result.
\begin{lemma}\label{notlimbl2}
Given $L>0$ there exist a small $b$ and a large $\mu(L)$ such that for $\mu\geq \mu(L)$, $\varphi_{\mu,\sigma} \in X$ and $I_\lambda(\varphi_{\mu,\sigma})<-L$ for any $\sigma\in Y_k$.
\end{lemma}

Remark that as a direct consequence of Lemma~\ref{notlimbl2}, $I_{\lambda}$ is not bounded from below. Moreover, by direct computations one can derive the following result.

\begin{lemma}\label{notlimbl3}
Let $\varphi_{\mu,\sigma}$ be defined in \eqref{phimu}. Then for any $\sigma \in Y_k$,
$$ \frac{e^{\varphi_{\mu,\sigma}}\chi_{S^+}}{\int_{S^+} e^{\varphi_{\mu,\sigma}}\,dV_{g_0}}  \weakto \sigma, \quad \text{as $\mu \to+\infty$}.$$
\end{lemma}

\

\subsubsection{\boldmath${\lambda\in(8\pi,16\pi)}$}

\

Let us consider the set $\Theta_\lambda$ introduced in \eqref{Thetalambda}, which is non contractible by Remark \ref{rem:Thetalambda}, both under the assumptions of Theorem~\ref{thm:equation1} or Theorem~\ref{thm:equation2}, namely if \ref{H4} or \ref{H5} holds.

Let us now map $\Theta_\lambda$ into arbitrary low sublevels of $I_\lambda$.\\
 Let $\tilde{\alpha}=\displaystyle{\max_{\{i\leq \ell \,|\,p_i\notin J_{\lambda}\}}} \alpha_i$ or $\tilde{\alpha}=0$ if $J_\lambda=\{p_1,\dots,p_\ell \}$ or $\ell=0$. For any $\alpha \in \left(\tilde{\alpha},\frac{\l}{8\pi} -1\right)$, $\mu>0$ and $p \in \Theta_{\lambda}$, we define

\beq\label{phimu2}
 \varphi_{\mu,p,\alpha}: \Sd \to \R, \quad \ \varphi_{\mu,p,\alpha}(x)=2 \log\left(\frac{\mu^{1+\alpha}}{1+(\mu \chi_{b}(dist(x,p)))^{2(1+\alpha)}}\right).
\eeq

\begin{lemma}\label{notlimb}
Given any $L>0$, there exist a small $b$ and a large $\mu(L)$ such that for any $\mu \geq \mu(L)$, $\varphi_{\mu,p,\alpha} \in X$, and $I_\lambda(\varphi_{\mu,p,\alpha})<-L$ for any $p\in \Theta_{\lambda}$.
\end{lemma}

\begin{proof}
The proof follows combining results of \cite{Djad,DjadliMalchiodi,MalchiodiRuizSphere}. We just point out that for $b<\frac{\theta}{4}$, where $\theta$ appears in the definition \eqref{Thetalambda} of $\Theta_\lambda$, $B_p(b)$ is compactly contained in $S^+$, thus we can use the detailed computations in \cite{MalchiodiRuizSphere} in order to estimate the logarithmic term of the functional.
\end{proof}

In particular, the previous result shows that $I_{\lambda}$ is unbounded from below. Moreover, as stated in the following result, the unit measures induced by $\tilde{K}e^{\varphi_{\mu,p,\alpha}}\chi_{S^+}$ concentrate around $p$ as $\mu\to+\infty$.

\begin{lemma}\label{notlimb22}
Let $\varphi_{\mu,p,\alpha}$ be defined in \eqref{phimu2}. Then for any $p\in\Theta_{\lambda}$,
$$ \frac{\tilde{K}e^{\varphi_{\mu,p,\alpha}}\chi_{S^+}}{\int_{S^+} \tilde{K}e^{\varphi_{\mu,p,\alpha}}\,dV_{g_0}}  \weakto \delta_p, \quad \text{as $\mu \to+\infty$}.$$
\end{lemma}
\begin{proof}
See Lemma 5.2 of \cite{MalchiodiRuizSphere}. Just minor modifications are needed.
\end{proof}

\

\section{Proofs of Theorems~\ref{thm:equation1}, \ref{thm:equation2} and \ref{nonexistence}}\label{sproof}
\setcounter{equation}{0}

In this section, we employ the previous results to prove Theorems~\ref{thm:equation1} and \ref{thm:equation2}. The proof is based on a min-max argument relying on the non trivial topology of the low sublevels of $I_\lambda$, which inherit the non contractibility of the sets $\Gamma_k$, $(P^{N^+})_k$ or $\Theta_\lambda$ (defined in Proposition \ref{prop:Psinoncontractible} and in \eqref{Thetalambda} respectively) depending on the cases.


At this point, we are ready to prove Theorems~\ref{thm:equation1} and \ref{thm:equation2}. Let us introduce some notations related to the min-max scheme in order to unify the proofs of Theorem~\ref{thm:equation1} and Theorem~\ref{thm:equation2}
$$
\mathcal{C} =
  \begin{cases}
    Y_k  & \text{if $\lambda_0\in(8\pi k,8\pi(k+1)), \, k\geq2,$} \\
    \Theta_{\lambda_0}   & \text{if $\lambda_0\in(8\pi,16\pi),$}
  \end{cases}
$$
where $\Theta_{\lambda_0}$ is defined in \eqref{Thetalambda} and
$$
Y=
\begin{cases}
P^{N^+}  & \text{if $N^+>k$,} \\
   \Gamma   & \text{if $S^+$ has a connected component which is non-simply connected },
  \end{cases}
$$
with $P^{N^+}$ and $\Gamma$ introduced in Proposition~\ref{prop:Psinoncontractible}.\\
Moreover, we set
$$
\omega =
  \begin{cases}
    \sigma \in Y_k & \text{if $\lambda_0\in(8\pi k,8\pi(k+1)), \, k\geq2$,}  \\
   p \in \Theta_{\lambda_0}      & \text{if $\lambda_0\in(8\pi,16\pi)$},
  \end{cases}
$$
  $$
\tilde{\varphi}_{\mu,\omega} =
  \begin{cases}
   \varphi_{\mu,\sigma}  & \text{if $\lambda_0\in(8\pi k,8\pi(k+1)), \, k\geq2$,}  \\
   \varphi_{\mu,p,\alpha}       & \text{if $\lambda_0\in(8\pi,16\pi)$}.
  \end{cases}
$$

Next, we define the topological cone $\hat{\mathcal C}$ over $\mathcal C$ as
$$
\hat{\mathcal{C}}=\left(\mathcal{C}\times\left[0,1\right]\right)/\left(\mathcal{C}\times\left\{1\right\}\right),
$$
where the equivalence relation identifies all the points in $\mathcal{C}\times\left\{1\right\}$. We will denote by $[\omega,t]$ an element of $\hat{\mathcal{C}}$, where $\omega\in\mathcal{C}$ and $t\in[0,1]$, and sometimes, with an abuse of notation, we will identify $[\omega,0]$ with $\omega$.

Then let us choose $\eps>0$ such that $(\lambda_0-\eps,\lambda_0+\eps)\subset(8\pi k,8\pi(k+1))$, $J_\lambda=J_{\lambda_0}$ and so $\Theta_\lambda=\Theta_{\lambda_0}$ for any $\lambda\in(\lambda_0-\eps,\lambda_0+\eps)$.
\\
Next, let us introduce the following class
$$
\mathcal{G}_{\mu,\lambda}=\{g:\hat{\mathcal{C}} \to X \, | \, \text{$g$ is continuous and }g([\omega,0])=\tilde{\varphi}_{\mu,\omega} \text{ for every $\omega\in\mathcal{C}$} \}.
$$
Notice that $\mathcal{G}_{\mu,\lambda}\neq \emptyset$, indeed if we fix $v\in X$ the map $g_v:\hat{\mathcal{C}}\to X$, defined as $g_v([\omega,t])=\log(te^v+(1-t)e^{\tilde{\varphi}_{\mu,\omega}})$, belongs to $\mathcal{G}_{\mu,\lambda}$.

Let us now fix $L>0$ so large such that both Proposition \ref{prop:Psinoncontractible} and Proposition \ref{prop:Psicontractible} apply and in turn $\mu>0$ so large that $I_\lambda(\tilde{\varphi}_{\mu,\omega})<-L$ for any $\omega\in\mathcal C$ and any $\lambda\in(\lambda_0-\eps,\lambda_0+\eps)$.

The latter choice is possible in view of Lemma \ref{notlimbl2} and Lemma \ref{notlimb}.

Next proposition will be crucial in our min-max argument.

\begin{proposition}\label{prophomot}
If $\alpha_1,\ldots,\alpha_m>0$ and under assumptions \ref{H1}, \ref{H2}, \ref{H3} and \ref{H4} or \ref{H5}, for any $\lambda\in(\lambda_0-\eps,\lambda_0+\eps)$ and any $g\in \mathcal{G}_{\mu,\lambda}$,  then the composition $\left.\Psi \circ g\right|_{\mathcal{C}}$ is homotopically equivalent to the identity map, where $\Psi$ is defined in Proposition \ref{prop:Psinoncontractible} for $k\geq2$ and Proposition \ref{prop:Psicontractible} for $k=1$. Moreover, $g(\mathcal{C}) $ is not contractible in $I_\lambda^{-L}$.
\end{proposition}

\begin{proof}
Let us introduce the homotopy \begin{align*}
H:&\left[\mu,+\infty\right) \times \mathcal{C} \longrightarrow \mathcal{C}\\
& \;\;\;\;\;\;\;\;\;\;(t,\omega) \;\;\longmapsto H(t,\omega)= \Psi \circ \tilde{\varphi}_{t,\omega}.
\end{align*}
Combining Lemma \ref{notlimbl3} with Proposition \ref{prop:Psinoncontractible} or Lemma \ref{notlimb22} with Proposition \ref{prop:Psicontractible} we obtain that
$$
H(t,\omega)\longrightarrow \omega\qquad\text{as $t\to+\infty$,}
$$
so $H$ realizes the desired homotopy equivalence.\\
In turn, by virtue of assumption \ref{H4} or \ref{H5} and our choice of $\eps$, $\mathcal{C}$ is not contractible, see Remark \ref{rem:PN+noncontr} and Remark \ref{rem:Thetalambda}. The above assertion implies easily that $g(\mathcal{C})$ is also not contractible.
\end{proof}

We now define the min-max value
$$
\overline{\mathcal{G}}_{\mu,\l}=\displaystyle{\inf_{g \in \mathcal{G}_{\mu,\l}} \sup_{z\in\hat{\mathcal{C}}} I_{\lambda}(g(z))}.
$$

\begin{lemma}\label{leminmax}
If $\alpha_1,\ldots,\alpha_m>0$ and under assumptions \ref{H1}, \ref{H2}, \ref{H3} and \ref{H4} or \ref{H5}, $\overline{\mathcal{G}}_{\mu,\l}\geq -L$ for any $\lambda\in(\lambda_0-\eps,\lambda_0+\eps)$.
\end{lemma}

\begin{proof}
For any $g \in \mathcal{G}_{\mu,\lambda}$, clearly $g(\mathcal{C})$ is contractible in $g(\hat{\mathcal{C}})$, being $\mathcal{C}$ contractible in $\hat{\mathcal{C}}$. On the other hand, by Proposition~\ref{prophomot}, $g(\mathcal{C})$ is not contractible in $I_{\lambda}^{-L}$, so that, $g(\hat{\mathcal{C}}) \nsubseteq I_{\lambda}^{-L}$, namely there exists $z \in \hat{\mathcal{C}}$ such that $I_{\lambda}(g(z)) > -L$.
\end{proof}

Consequently Lemma~\ref{leminmax} together with our choice of $L$ and $\mu$ imply that $\overline{\mathcal{G}}_{\mu,\l}>\max \{I_{\l}(g_{\mu}(z)) : \, z \in \mathcal{C} \}$ and this provides a min-max structure. Unfortunately, the Palais--Smale condition is not known to hold for $I_{\lambda}$ and to overcome this difficulty we will use the well-known monotonicity method of Struwe, introduced firstly in \cite{Struwe}. Since this argument has been applied many times even for this functional, as for instance in \cite{Djad,DjadliMalchiodi,MalchiodiRuizSphere}, we will be sketchy. A starting point to apply this trick is the following easy monotonicity result.

\begin{lemma}\label{monomini}
The function $\lambda \mapsto \frac{\overline{\mathcal{G}}_{\mu,\l}}{\lambda}$ is monotonically decreasing.
\end{lemma}

\begin{proof}
Just observe that, for $\l < \l'$,
$$ \frac{I_{\l}(u)}{\l}-\frac{I_{\l'}(u)}{\l'} = \frac{1}{2} \left( \frac{1}{\l} - \frac{1}{\l'} \right) \int_{\Sd} \left| \nabla u \right|^2\,dV_{g_0} \geq 0. $$

Since $\overline{\mathcal{G}}_{\mu,\l}$ is a min-max value for $I_\l$, the previous estimate implies the monotonicity of $\frac{\overline{\mathcal{G}}_{\mu,\l}}{\lambda}$.
\end{proof}

\begin{proof}[Proofs of Theorem~\ref{thm:equation1} and Theorem~\ref{thm:equation2}] Let $\lambda_0\in(8\pi k,8\pi(k+1))\setminus\Gamma(\diva_\ell)$, for some $k\geq1$ and consider $\lambda\in(\lambda_0-\eps,\lambda_0+\eps)$ as above.
By Lemma \ref{monomini} we obtain that the set
$$
E_k=\left\{ \l \in  (\lambda_0-\eps,\lambda_0+\eps):\ \mbox{ the map } \l \mapsto \overline{\mathcal{G}}_{\mu,\l} \mbox{ is differentiable at} \ \l\right\}.
$$
is dense in $(\lambda_0-\eps,\lambda_0+\eps)$. Moreover, for any $\l \in E_k$, there exists a sequence $u_n \subset X$ which is bounded in $H^1(\Sd)$, $I_{\l}(u_n) \to  \overline{\mathcal{G}}_{\mu,\l}$ and $I_{\l}'(u_n) \to 0$, and in turn this implies that for any $\lambda \in E_k$ there exists $u_\lambda$ critical point of $I_{\lambda}$. Indeed since $u_n$ is bounded, up to a subsequence, $u_n \weakto u_\l$ and standard arguments show that $u_n \to u_\l$ strongly where $u_\l$ is a critical point for $I_\l$.

Finally, since we are assuming $(\Sigma,g)=(\Sd,g_0)$ we can apply Theorem~\ref{thm:compactness} and find a solution $u_{\lambda_0}$ to \eqref{equation} with $\lambda=\lambda_0$.\\ 
This concludes the proofs of Theorem~\ref{thm:equation1} and Theorem~\ref{thm:equation2}.

\end{proof}

\

At last we prove the non existence result, namely Theorem~\ref{nonexistence}.

Given a point $p\in \Sd$ and a function $F \in \mathcal{F}_p$, defined in \eqref{setRP}, the strategy to prove the Theorem~\ref{nonexistence} is to construct a function $K_F$ defined in $\Sd$ such that the stereographic projection of $\tilde{K}=Ke^{-h_1}$ in $\mathbb{R}^2$ is a radial function which verifies the monotonicity condition \eqref{NE} and then apply Theorem \ref{nethm:ChenLi}.

\begin{proof}[Proof of Theorem~\ref{nonexistence}]
Let us fix a function $F=F(\varphi)$ in $\mathcal{F}_p$ expressed in spherical coordinates, where without loss of generality we can suppose $p=(0,0,1)$. 
Let $h$ be the regular part of the function $h_1$ introduced in \eqref{accame} and define
\beq\label{KR}
K_F(\varphi)=F(\varphi)e^{h(\varphi)}g_\lambda(P(\varphi)) \quad \mbox{with $\varphi \in \left(0,\pi \right]$ }, \quad K_F(0)=0,
\eeq
where $P:(0,\pi]\to\R^2$ is the stereographic projection of $\Sd$ into $\R^2$ and $g_\lambda(y)=\left(\dfrac4{(1+|y|^2)^2}\right)^{\frac{\lambda}{8\pi}-1}$ for $y\in\R^2$.
By \eqref{KR} we have that
$$
\tilde{K}_F(\varphi)=K_F(\varphi)e^{-h_1(\varphi)}= F(\varphi)\varphi^{2\alpha}g_\lambda(P(\varphi)) \quad \mbox{with $\varphi \in \left(0,\pi \right]$ },
$$
where $\log(\varphi)^{2\alpha}$ corresponds to the singular part of $-h_1$ in spherical coordinates.

Now, as done in Section~\ref{Scom}, by means of the stereographic projection we transform \eqref{equation} into
$$
-\Delta v  = \hat{K}_{F,\lambda}\, e^{v}  \qquad \text{in $\mathbb{R}^2$,}
$$
where $v$ satisfies \eqref{ieqv} and
\[
\hat{K}_{F,\lambda}(y)=\tilde K_F(P^{-1}(y))\lambda g_{\lambda}^{-1}(y)=\lambda F(P^{-1}(y))(P^{-1}(y))^{2\alpha}
\]
is bounded and verifies condition \eqref{NE}, being $F\in\mathcal F_p$.\\ 
At last to conclude it suffices to apply Theorem ~\ref{nethm:ChenLi} with $R=\hat{K}_{F,\lambda}$.
\end{proof}

\subsection*{Acknowledgements}
F.D.M. has been supported by FIRB project \textquotedblleft{Analysis and Beyond}\textquotedblright and PRIN $201274$FYK7$\_005$ and R.L.-S. by the Mineco Grant MTM2015-68210-P and by J. Andalucia (FQM116). During the preparation of this work F.D.M. was hosted by University of Granada, while R.L.-S. spent a long period as visiting researcher at University of Rome Tor Vergata, they are grateful to these institutions for the kind hospitality and thank D.Ruiz and G.Tarantello for the invitations. Both authors are grateful to D.Bartolucci, W.Chen, C.Li, D.Ruiz and G.Tarantello for their suggestions and discussions concerning the subject. Last but not least they warmly thank the referee for his accurate and precious review.

\end{document}